\documentclass[a4paper]{amsart}
\usepackage[latin1]{inputenc}
\usepackage{amsfonts}
\usepackage{amsmath,latexsym,amssymb,amsfonts, amsthm}
\usepackage{esint}
\usepackage{cite}
\usepackage{graphicx}
\usepackage{amscd}
\usepackage{color}
\usepackage{bm}             
\usepackage{enumerate}
\usepackage[dvips]{epsfig}
\usepackage{psfrag}
\usepackage{epsfig}

\usepackage[
  hmarginratio={1:1},     % equal left and right margins
  vmarginratio={1:1},     % equal top and bottom margins=
  textwidth=15cm,        % new text width
  textheight=21cm,
  heightrounded,          % always useful
]{geometry}

\usepackage{graphicx,color}
\usepackage[colorlinks]{hyperref}
\hypersetup{linkcolor=blue,citecolor=blue,filecolor=black,urlcolor=blue}

\newtheorem{theorem}{Theorem}
\newtheorem{corollary}[theorem]{Corollary}
\newtheorem{proposition}[theorem]{Proposition}
\newtheorem{lemma}[theorem]{Lemma}
\theoremstyle{definition}
\newtheorem{remark}{Remark}

\numberwithin{theorem}{section}

\numberwithin{remark}{section}

\numberwithin{equation}{section}

\newcommand{\R}{\mathbb{R}}
\newcommand{\N}{\mathbb{N}}

\newcommand{\cE}{{\mathcal E}}   
   
\newcommand{\cG}{{\mathcal G}}

\newcommand{\weak}{\rightharpoonup}

\newcommand{\eps}{\varepsilon}

\newcommand{\beq}{\begin{equation}}
\newcommand{\eeq}{\end{equation}}

\DeclareMathOperator{\loc}{loc}

\newcommand{\Gcal}{{\mathcal{G}}}

\renewcommand{\a}{{\alpha}}

\newcommand{\lp}[2]{\|#1\|_{L^{#2}(\cG)}}

\title[Ground states for the combined NLS on graphs]{Ground states for the NLS equation with combined nonlinearities on non-compact metric graphs}

\author[D. Pierotti]{Dario Pierotti}\thanks{}
\address{Dario Pierotti \newline \indent
Dipartimento di Matematica,  Politecnico di Milano,  \newline \indent
Via Edoardo Bonardi 9, 20133 Milano, Italy}
\email{dario.pierotti@polimi.it}

\author[N. Soave]{Nicola Soave}\thanks{}
\address{Nicola Soave \newline \indent
Dipartimento di Matematica,  Politecnico di Milano,  \newline \indent
Via Edoardo Bonardi 9, 20133 Milano, Italy}
\email{nicola.soave@gmail.com; nicola.soave@polimi.it}

\thanks{Nicola Soave is partially supported by the INDAM-GNAMPA group.}

\begin{document}

\begin{abstract}
We investigate the existence of ground states with prescribed mass for the NLS energy with combined $L^2$-critical and subcritical nonlinearities, on a general non-compact metric graph $\cG$. The interplay between the different nonlinearities creates new phenomena with respect to purely critical or subcritical problems on graphs; from a different perspective, topological and metric properties of the underlying graph drastically influence existence and non-existence of ground states with respect to the analogue problem on the real line.
\end{abstract}

\maketitle

\section{Introduction}

We investigate the existence of ground states for the NLS energy with combined nonlinearities
\beq\label{def E}
E_\a(u, \cG) := \int_{\cG} \left(\frac12 |u'|^2 - \frac16|u|^6 - \frac{\alpha}{p}|u|^p\right),
\eeq
under the mass constraint 
\[
u \in H^1_\mu(\cG) := \left\{ u \in H^1(\cG): \ \int_{\cG} |u|^2 = \mu \right\},
\]
where $\a \in \R$, $p \in (2,6)$, and $\cG$ is a non-compact metric graph. A metric graph is a connected metric space obtained by glueing together a finite number of closed line intervals, the \emph{edges} of the graph, by identifying some of their endpoints. The endpoints are the \emph{vertices} of the graph. Any bounded edge $\textrm{e}$ is identified with a closed bounded interval $[0,\ell]$ (where $\ell$ is the length of $\textrm{e}$), while unbounded edges are identified with (a copy of) the closed half-line $[0,+\infty)$. We always consider graphs with a finite number of vertices and edges, and hence the requirement that $\cG$ is non-compact translates into the existence of at least one unbounded edge. It would also be possible to consider non-compact graphs with infinite number of vertices and/or edges, and no unbounded edge (such as grids, or trees); however, this different class of graphs will not be considered here (we refer the interested reader to \cite{ADST19, ADR19, DST19}).

In this framework, by \emph{a ground state of mass $\mu$} we mean a minimizer for the problem
\beq\label{def gs level}
\mathcal{E}_\alpha(\mu,\cG):= \inf_{u \in H^1_\mu(\cG)} E_\alpha(u, \cG).
\eeq
The value $\mathcal{E}_\alpha(\mu,\cG)$ is called \emph{ground state energy level}, and it is clear that, searching for ground states, it is sufficient to work with real valued functions. \\
Ground states, and more in general critical points of $E_\alpha(\cdot,\Gcal)$ constrained on $H^1_{\mu}(\Gcal)$, satisfy, for some $\lambda \in \R$, the stationary NLS equation
\begin{equation}\label{eq:NLSE}
-u'' +  \lambda u = |u|^4 u + \alpha |u|^{p-2} u
\end{equation}
on every edge; moreover, at each vertex the Kirchhoff condition is satisfied, which requires
 the sum of all the outgoing derivatives to vanish (see \cite[Proposition 3.3]{AST1}).
Through the usual ansatz $\Phi(x,t) = e^{-i\lambda t} u(x)$, such critical points correspond to solitary wave solutions to the \emph{Schr\"odinger equation with combined nonlinearities}
\[
i \partial_t \Phi  + \partial_{xx} \Phi + |\Phi|^{q-2}\Phi + \alpha |\Phi|^{p-2}\Phi = 0, 
\qquad
x\in\Gcal, \ t>0,
\]
with $q=6$. Starting from the seminal contribution by T. Tao, M. Visan and X. Zhang \cite{TaoVisZha}, the study of this equation (for general $2<p<q$) in the Euclidean space $\R^N$ attracted much attention: global well-posedness, scattering, the occurrence of blow-up and more in general dynamical
properties were studied in \cite{TaoVisZha} and many papers \cite{AkaIbrKikNaw, ChMiZh, Feng, FukOht, GuZu, KiOhPoVi, JJTV, LeCozMaRa, MiaXuZha, MiaZhaZhe, Zha} (see also the references therein). Recently, the existence of normalized ground states was studied in \cite{So1, So2} (where a more general notion of ground state, suited to deal with a wider range of exponent than ours, is considered). On the contrary, the existence of ground states on metric graphs is completely open, and is the main topic of this paper.

When dealing with the $1$-dimensional Schr\"odinger equation, it is well known that the exponent $q=6$ plays a special role. In connection to minimization of the NLS energy, the homogeneous energy 
\[
\int_{\R} \left(\frac12 |u'|^2 - \frac1q|u|^q\right)
\]
with $q=6$ is bounded from below on $H^1_\mu(\R)$ if and only if the mass $\mu$ is smaller than or equal to $\mu_{\R} := \pi \sqrt{3}/2$, and a ground state of mass $\mu$ exists if and only if $\mu=\mu_{\R}$. This restriction is a consequence of the fact that the two terms in the energy scale in the same way with respect to mass-preserving dilations, a phenomenon which makes the minimization very unstable. For this reason, the problem is called \emph{$L^2$-critical}. If instead $2<q<6$, then the energy functional is always bounded from below on $H^1_\mu(\R)$, and a ground state of mass $\mu$ exists for every positive $\mu$. And, finally, if $q>6$, then the energy functional is always unbounded from below on $H^1_\mu(\R)$.

An interesting feature of the equation with combined nonlinearities stays in the fact that the presence of two powers destroys the scale invariance of the homogeneous critical equation. In terms of problem \eqref{def gs level}, this translates into existence of ground states for $\cE_\a(\mu, \R)$ if and only if $\mu$ stays in the interval of masses $(0,\mu_{\R})$, when $\alpha>0$. On the contrary, for $\alpha<0$ there is no ground state at all (for any $\mu>0$), see \cite[Theorem 1]{So1} (or Section \ref{sec: pre} below).

On general graphs the situation is much more involved, and the existence or non-existence of ground states in relation with the topological and metric properties of the graph is a highly non-trivial issue, even for subcritical problems. A systematic study for the homogeneous energy was carried out by R. Adami, E. Serra and P. Tilli in a series of paper concerning both the subcritical \cite{AST1, AST3} and the critical \cite{AST2} case. The aim of this paper is to study the inhomogeneous problem \eqref{def gs level}, analyzing the impact of the subcritical ``perturbation" $\alpha \lp{u}{p}^p/p$ on the homogeneous critical energy. This is somehow in the spirit of the Brezis-Nirenberg problem: we have a critical problem for which the structure of the ground states is known, thanks to \cite{AST2}, and we study what happens when we add a lower order perturbation. As we shall see, the interplay between the different powers creates new phenomena with respect to purely critical or subcritical problems. We emphasize that, while we always suppose that the critical nonlinearity is of focusing type (i.e., the coefficient in front of $\lp{u}{6}^6$ in the energy is negative), we allow for both focusing ($\alpha>0$) and defocusing ($\alpha<0$) lower order term. The results will be very different in the two cases. It would have also been possible to consider a defocusing critical nonlinearity, but in such case the energy is bounded from below for every choice of mass $\mu>0$, making the problem more similar to the purely subcritical one (and hence less interesting from the point of view of the present investigation).

The study of nonlinear Schr\"odinger equations on metric graphs has attracted considerable attention in the last decade. From the mathematical point of view, the problem presents a number of interesting new features with respect to the classical Euclidean setting. Furthermore, nonlinear evolution on graphs turns out to be relevant also from the physical point of view (see e.g. \cite{Meh, BoCa, SMSS}). We do not attempt to provide a complete overview of the many available results in the literature, for which we refer the interested reader to \cite{AST11, No} and the references therein. 

We limit to mention that results related to ours, concerning existence and non-existence of ground states for subcritical and critical problems, have been obtained in \cite{AST1, AST3, AST2}, whose main results will be discussed in details in what follows; in \cite{ACFN14, ACFN12} where the particular case of the star-graph is discussed (see also the reference therein); in \cite{DTcalcvar, ST16}, devoted to subcritical and critical problems with localized nonlinearities; in \cite{PSV, NP}, regarding existence of critical points (not necessarily ground states) of the critical NLS energy; in \cite{Do16, CDS18}, regarding the minimization of the energy on compact graphs. See also the references therein. 

Moreover, very recently, a problem with combined nonlinearity on $\R$ or on the star-graph was studied in \cite{BD19, ABD20}. The combined nonlinearity in \cite{BD19, ABD20} is of different nature with respect to the one considered here, being obtained by summing the $L^p$ norm of $u$ to the point-wise value $|u(0)|^q/q$, with $p$ and $q$ both subcritical. Also in \cite{BD19, ABD20} the interplay between the two nonlinearities gives raise to new phenomena with respect to the homogeneous case.
%They obtained different results for different ranges of the parameters $p$ and $q$, showing how the presence of a combined nonlinearity drastically changes the structure of the ground states.

\subsection{Statement of the main results}

In order to state our main results in a precise form, it is convenient to briefly review some results concerning the homogeneous problem obtained choosing $\alpha=0$ in \eqref{def E} and \eqref{def gs level}. This problems was studied in \cite{AST2}, where the authors showed that for any non-compact graph $\cG$ there exists a critical mass $\mu_{\cG}>0$ for which the following holds (see \cite[Proposition 2.4]{AST2}):
\begin{itemize}
\item[($i$)] If $\mu \le \mu_{\cG}$, then $\cE_0(\mu, \cG) = 0$, and the infimum is not attained for $\mu < \mu_{\cG}$;
\item[($ii$)] If $\mu > \mu_{\cG}$, then $\cE_0(\mu, \cG) < 0$ (possibly $-\infty$);
\item[($iii$)] If $\mu > \mu_{\R}$, then $\cE_0(\mu, \cG) =-\infty$.
\end{itemize}
Moreover, $\mu_{\R^+} \le  \mu_{\cG} \le \mu_{\R}$, where $\mu_{\R^+}$ (the critical mass associated with the graph $\R^+$) is equal to $\mu_{\R}/2$. Next, in order to sharpen their results, the authors of \cite{AST2} identified 4 different (mutually exclusive) classes of graphs (see \cite[Section 3]{AST2}):

\begin{enumerate}
 \item \emph{$\Gcal$ has at least a terminal point.} A terminal point (a tip) is a vertex $\textrm{v}$, not at infinity, of degree 1 (that is, there is only one edge having $\textrm{v}$ has an extremum). In this case $\mu_\Gcal = \mu_{\R^+}$, $\cE_0(\mu, \cG) =-\infty$ for $\mu > \mu_{\R^+}$, and $\cE_0(\mu, \cG)$ is never achieved unless $\Gcal$ is isometric to $\R^+$ and $\mu = \mu_{\R^+}$;
 \item \emph{$\Gcal$ admits a cycle covering.} Here ``cycle'' means either a bounded loop, or an unbounded path joining two distinct points at infinity. Equivalently, $\Gcal$ has at least two half-lines and no terminal point, and whenever $\Gcal\setminus \textrm{e}$ has two connected components, both are unbounded (here $\textrm{e}$ denotes any bounded edge). Then $\mu_\Gcal = \mu_{\R}$, and $\cE_0(\mu, \cG)$ is never achieved unless $\Gcal$ is isometric to $\R$ or to a ``tower of bubbles'' (see \cite[Example 2.4]{AST1}), and $ \mu= \mu_{\R}$;
 \item \emph{$\Gcal$ has exactly one half-line and no terminal point.} Then $\mu_\Gcal =
 \mu_{\R^+}$, and $\cE_0(\mu, \cG) \in (-\infty,0)$ is achieved if and only if $\mu\in (\mu_{\R^+},\mu_{\R}]$;
 \item \emph{$\Gcal$ does not belong to any of the previous three cases, i.e. it has no tip, no cycle-covering and at least 2 half-lines.} Then, assuming further that
 $\mu_\Gcal < \mu_{\R}$, we have that $\cE_0(\mu, \cG) \in (-\infty,0)$
 is achieved if and only if $\mu\in[\mu_{\Gcal},\mu_{\R}]$; if $\mu_\Gcal = \mu_{\R}$, then nothing is known.
\end{enumerate}

For graphs of the fourth type, it is also known that $\mu_{\cG}$ is strictly larger than $\mu_{\R^+}$, see \cite{PSV}, and there are explicit examples for which $\mu_{\cG}<\mu_{\R}$ (on the contrary, it is an open problem to find a graph of type 4 having critical mass equal to $\mu_{\R}$).

From the above discussion, it emerges the key role of the critical mass $\mu_{\cG}$.

Let us now consider the inhomogeneous case $\alpha \neq 0$. In this case, for any non-compact graph $\cG$, we introduce a new critical mass, defined as follows:
\beq\label{new crit}
\tilde \mu_{\cG}:= \begin{cases} \mu_{\R^+} & \text{if $\cG$ has a terminal point}, \\  \mu_{\R} & \text{if $\cG$ does not have a terminal point}. \end{cases} 
\eeq
Notice that, by the classification in \cite{AST2}, $\tilde \mu_{\cG} =  \mu_{\R^+} = \mu_{\cG}$ if $\cG$ has a terminal point, while $\tilde \mu_{\cG} = \mu_{\R} \ge \mu_{\cG}$ in the other cases, with strict inequality for graphs of type 3 and some graphs of type 4.

\subsection*{Focusing lower order term: $\alpha>0$.} At first, we consider the case when the lower order term in \eqref{def E}  is of focusing type, i.e. $\alpha>0$. The first of our main results is the following. %Thanks to Proposition \ref{prop: basic}, we restrict ourselves in the range $\mu \in (0,\mu_{\R})$, since for $\mu \ge \mu_{\R}$ we already know that $\cE_\a(\mu, \cG) =-\infty$. 

\begin{theorem}\label{thm: main foc}
Let $\cG$ be a non-compact metric graph, $p \in (2,6)$, $\alpha>0$. We have that
\[
\cE_\alpha(\mu, \cG)  \in (-\infty,0) \quad \text{if $\mu \in (0,\tilde \mu_{\cG})$}, \quad \cE_\alpha(\mu, \cG)  = -\infty \quad \text{if $\mu \ge \tilde \mu_{\cG}$}.
\]
Moreover, if $\mu \in (0,\tilde \mu_{\cG})$ and 
\[
\cE_\a(\mu, \cG) < \cE_\a(\mu, \R),
\]
then there exists a ground state for $\cE_\a(\mu, \cG)$.
\end{theorem}

\begin{corollary}\label{cor: foc}
Let $\cG$ be a non-compact metric graph, $p \in (2,6)$, $\alpha>0$ and $\mu \in (0,\tilde \mu_{\cG})$. If there exists $u \in H^1_\mu(\cG)$ such that 
\[
E_\a(u, \cG) \le \cE_\a(\mu, \R),
\]
then there exists a ground state for $\cE_\a(\mu, \cG)$.
\end{corollary}

These statements already unveils how the combination of subcritical and critical powers mixed things up: as in the critical case, we have a critical mass above which the energy functional is unbounded from below, and hence ground states do not exist. Below this threshold, the picture is somehow analogue to the one available in the focusing subcritical problem, as described in \cite[Section 3]{AST3} (see in particular Theorem 3.3 and Corollary 3.4 therein).

Several applications of this corollary are available, essentially replicating the arguments in \cite[pag. 213-214]{AST3} in the present setting. In particular, supposing $\alpha>0$, we have that:
\begin{itemize}
\item[(a)] If $\cG$ consists of two half-lines and one bounded edge (or arbitrary length) joined at their initial point, then there exists a ground state for $\cE_{\alpha}(\mu, \cG)$ if and only if $\mu \in (0, \mu_{\R^+})$. Moreover, $\cE_{\alpha}(\mu, \cG) =-\infty$ for $\mu \ge \mu_{\R^+}$.
\item[(b)] If $\cG$ is isometric to $\R$ or to a tower of bubbles, then there exists a ground state for $\cE_{\alpha}(\mu, \cG) \in (-\infty,0)$ if and only if $\mu \in (0, \mu_{\R})$. Moreover, $\cE_{\alpha}(\mu, \cG) =-\infty$ for $\mu \ge \mu_{\R}$.
\item[(c)] If $\cG$ is a tadpole graph (a half-line attached to a self-loop), then there exists a ground state for $\cE_{\alpha}(\mu, \cG)$ if and only if $\mu \in (0, \mu_{\R})$. Moreover, $\cE_{\alpha}(\mu, \cG) =-\infty$ for $\mu \ge \mu_{\R}$.
\item[(d)] If $\cG$ is a sign-post graph (two-half-lines and a bounded edge $\textrm{e}$ glued in the same vertex, and a further self-loop attached to the second vertex of $\textrm{e}$), then there exists a ground state for $\cE_{\alpha}(\mu, \cG)$ if and only if $\mu \in (0, \mu_{\R})$. Moreover, $\cE_{\alpha}(\mu, \cG) =-\infty$ for $\mu \ge \mu_{\R}$.
\end{itemize}
The possibility of replicating the same arguments used in \cite{AST3} is a consequence of the fact that, when $\alpha>0$ and $\mu \in(0,\mu_{\R})$, there exists a ground state $\phi_{\alpha,\mu}$ for $\cE_\alpha(\mu,\R)$, and $\phi_{\alpha,\mu}$ has the same properties of ground states for the energy functional associated with the subcritical homogeneous NLS energy on $\R$: in particular, it is even with respect to a point and decreasing from that point on (see Section \ref{sec: pre} for a detailed discussion on the problem on $\R$ and $\R^+$). This is all what is needed in order to apply Corollary \ref{cor: foc} in the above cases, precisely in the same way as \cite[Corollary 3.4]{AST3}  is applied in \cite[pag. 213, 214]{AST3}.

Notice that the particular graphs described in examples (a), (b), (c), (d) belong to type (1), (2), (3), (4) respectively. Therefore, for each class of graph we may have existence of ground states in a full interval of masses, in sharp contrast to what happens in the homogeneous critical case.

\medskip

Further results can be obtained by restricting to specific classes of graphs, in the spirit of \cite{AST1, AST3}. We analyze at first graphs which can be covered by cycles. As proved in \cite[Theorem 2.5]{AST1} and \cite[Theorem 3.2]{AST2}, this class of graph is particularly unfavorable for the existence of ground states in the homogeneous subcritical or critical cases. The same happens with combined focusing nonlinearities. 

\begin{theorem}\label{thm: cycle}
Let $\cG$ be a non-compact metric graph which admits a cycle covering, and let $p \in (2,6)$ and $\alpha>0$. Then $\cE_\a(\mu,\cG)>-\infty$ if and only if $\mu \in (0, \mu_{\R})$, and for these masses the infimum is achieved if and only if $\cG$ is isometric to $\R$ or to a tower of bubbles.
\end{theorem}

Notice that, even if $\cE_\a(\mu,\cG)>-\infty$, a ground state does not necessarily exist, in contrast with the case $\cG = \R$ (see Theorem \ref{thm: R} below).%Recalling that if $\alpha<0$ graphs with cycle covering never have ground states (see Proposition \ref{prop: basic}), this result provide a complete classification in the class of graphs with cycle covering. 

\medskip 

We now move to graphs with a terminal point. By Theorem \ref{thm: main foc}, the energy is unbounded from below on $H^1_\mu(\cG)$ whenever $\mu \ge \mu_{\R^+}$. If instead $\mu \in (0,\mu_{\R^+})$, the energy is bounded from below and it makes sense to search for ground states (for instance, see example (a) above). In this range we can prove that existence or non-existence of ground states may depend on metric properties of $\cG$. This is again in analogy  with the subcritical homogeneous problems, see \cite[Section 4]{AST3}.

\begin{proposition}\label{prop: ex tip}
Let $\cG$ be a non-compact metric graph with a terminal edge of length $\ell$. Let $\alpha>0$ and $\mu \in (0,\mu_{\R^+})$. There exists $\bar \ell= \bar \ell(\alpha,\mu)>0$ such that, if $\ell \ge \bar \ell$, then there exists a ground state for $\cE_\a(\mu, \cG)$.
\end{proposition}

In the opposite direction:

\begin{proposition}\label{prop: non ex tip}
Let $\mu \in (0,\mu_{\R^+})$, $\alpha>0$, and let $\cG$ be a non-compact graph such that $\cE_\a(\mu, \cG)$ does not admit a ground state (for instance, this is the case if $\cG$ can be covered by cycle and is not isomorphic to $\R$ or to a tower of bubbles). Let $\cG_\ell$ denote the graph made up by glueing a terminal edge of length $\ell$ at a fixed vertex of $\cG$. Then there exists $\tilde \ell = \tilde \ell(\alpha,\mu) >0$ such that, if $\ell < \tilde \ell$, then there is no ground state for $\cE_\a(\mu, \cG_{\ell})$.
\end{proposition}

\begin{remark}
Let $\cG$ be a star-graph, that is a graph made of $N \ge 3$ half-lines glued together at their common origin. This class of graphs represents a prototypical non-trivial model for non-compact graphs with a cycle covering, and for this reason their study attracted a lot of attention in the last decade (see \cite{ACFN12, ACFN14} and references therein). By Theorem \ref{thm: cycle}, $\cE_\a(\mu, \cG)$ does not admit a ground state, for any $\mu \in (0,\mu_{\R^+})$. Thus, Propositions \ref{prop: ex tip} and \ref{prop: non ex tip} imply that ground states on $\cG_{\ell}$ do exist if the length $\ell$ of the terminal edge is larger than $\bar \ell$, and do not exist if $\ell < \tilde \ell$.  

It would be interesting to prove that $\tilde \ell = \bar \ell$ (this is the case for the homogeneous problem, see \cite[Theorem 4.4]{AST3}).
\end{remark}

\begin{remark}
One can compare the above statements with the main results in \cite{AST1, AST3} (more specifically, cf. Theorem \ref{thm: main foc} and \cite[Theorem 3.3]{AST3}, Corollary \ref{cor: foc} and \cite[Corollary 3.4]{AST3}, 
Theorem \ref{thm: cycle} with \cite[Theorem 2.5]{AST1}, Proposition \ref{prop: ex tip} with \cite[Proposition 4.1]{AST3}, and Proposition \ref{prop: non ex tip} with \cite[Theorem 4.4]{AST3}). As anticipated, it emerges that for the NLS energy \eqref{def E} in the focusing case $\alpha>0$ critical and subcritical effects are combined in the following way: as in the critical case, there exists a critical threshold $\tilde \mu_{\cG}$ for the mass, above which the ground state energy level is $-\infty$. For masses below $\tilde \mu_{\cG}$, ground states may exist or not, and subcritical methods can be often adapted to answer this question.
\end{remark}

%{\color{red} Da fare: - sarebbe interessante arrivare a una caratterizzazione precisa nel caso della proposizione \ref{prop: non ex tip}, del tipo c'\`e un ground state se e solo se il terminal edge ha lunghezza $\ge \bar \ell$. \\
%- Sarebbe interessante dire qualcosa su grafi con una solo semiretta, nello spirito di \cite[Sezione 5]{AST3}. L\`i usano abbastanza propriet\`a di scaling, che qui non abbiamo.}

\subsection*{Defocusing lower order term: $\alpha<0$}

The defocusing case presents, again, the combination of critical and subcritical effects, but in a quite different way with respect to the focusing one. A preliminary result is the following.

\begin{proposition}\label{prop: basic intro}
Let $\cG$ be a non-compact metric graph, $p \in (2,6)$, $\alpha<0$. Then
\[
\cE_\alpha(\mu, \cG) = 0 \quad \text{if $\mu \in (0, \mu_\cG]$}, \quad \text{and}\quad \cE_\alpha(\mu, \cG) =-\infty \quad \text{if $\mu > \mu_\R$}
\]
Moreover, the infimum is never achieved when $\mu \in (0,\mu_{\cG}]$.
\end{proposition}

Notice that $\cE_\alpha(\mu, \cG)$ is never attained when $\mu \le \mu_{\cG}$ and $\alpha<0$, while the same range of masses is favorable for existence when $\alpha>0$. 

From Proposition \ref{prop: basic intro}, in studying the minimization problem $\cE_\a(\mu, \cG)$ we can focus on the range $(\mu_{\cG}, \mu_{\R}]$. According to the classification in \cite{AST2}, this interval is empty if $\cG$ is of type (2), and may be empty also for some graphs of type (4); plainly, for these classes of graphs ground states never exist. Concerning graphs of type (1), (3), and graphs of type (4) with $\mu_{\cG}<\mu_{\R}$, we have the following:

\begin{theorem}\label{thm: main def}
Let $\cG$ be a non-compact metric graph, $p \in (2,6)$, $\alpha<0$. Then the following alternative occurs:
\begin{itemize}
\item[($i$)] if $\mu_{\cG} = \tilde \mu_{\cG} = \mu_{\R^+}$, then $\cE_\a(\mu,\cG) = -\infty$ for every $\mu \in (\mu_{\cG}, \mu_{\R}]$ (and, plainly, a ground state does not exist).
\item[($ii$)] If $\mu_{\cG} < \tilde \mu_{\cG} = \mu_{\R}$, then for every $\mu \in  (\mu_{\cG}, \mu_{\R}]$ there exists $ \bar \alpha<0$ depending on $\mu$ and $\cG$ (possibly equal to $-\infty$) such that:
\begin{itemize}
\item[($a$)] if $\alpha \in (\bar \alpha,0)$, then $\cE_\a(\mu,\cG) \in (-\infty,0)$, and the infimum is achieved;
\item[($b$)] if $\alpha < \bar \alpha$, then $\cE_{\alpha}(\mu, \cG) = 0$, and the infimum is not achieved.
\end{itemize}
Furthermore, if $\mu<\mu_{\R}$, then $\bar \alpha>-\infty$, and $\cE_{\bar \alpha}(\mu,\cG) = 0$.\end{itemize}
\end{theorem}

It is an interesting open problem to establish whether there are ground states for $\cE_{\bar \alpha}(\mu,\cG) = 0$ or not, and to understand if $\bar \alpha>-\infty$ also when $\mu = \mu_{\R}$.

%
%\begin{theorem}\label{thm: main def tip}
%Let $\cG$ be a non-compact metric graph with a terminal edge, $p \in (2,6)$, $\alpha<0$. Then $\cE_\a(\mu, \cG) = -\infty$ for every $\mu>\mu_{\R^+}=\mu_{\cG}$.
%\end{theorem}
%
%
%%assuming $\alpha<0$, by Proposition \ref{prop: basic} ground states with mass $\mu$ may exist only if $\mu_{\cG}< \mu_{\R}$ and $\mu \in (\mu_{\cG}, \mu_{\R}]$. For this range we have the following statement:
%
%\begin{theorem}\label{thm: main def no tip}
%Let $\cG$ be a non-compact metric graph without terminal edges, $p \in (2,6)$, $\alpha<0$. Suppose that $\mu_{\cG}<\mu_{\R}$, and let $\mu \in (\mu_{\cG}, \mu_{\R}]$. Then there exists $0<\bar \alpha_1 \le \bar \alpha_2$ (depending on $\mu$ and $\cG$) such that:
%\begin{itemize}
%\item[($i$)] if $|\alpha|< \bar \alpha_1$, then $\cE_\a(\mu,\cG) \in (-\infty,0)$, and the infimum is achieved;
%\item[($ii$)] if $|\alpha| \ge \bar \alpha_2$, then $\mu \in (\mu_{\cG}, \tilde \mu_{\cG}]$, then $\cE_{\alpha}(\mu, \cG) = 0$, and it is not achieved.
%\end{itemize}
%\end{theorem}

\begin{remark}\label{rmk: main def}
Recalling again the classification in \cite{AST2} and the definition of $\tilde \mu_{\cG}$, we can better describe Theorem \ref{thm: main def}. If $\cG$ has a terminal edge, then $\mu_{\cG} = \tilde \mu_{\cG}= \mu_{\R^+}$, and hence alternative ($i$) takes place: again, we have non-existence of ground states for any mass. If instead $\cG$ is a graph of type (3), or of type (4) with $\mu_{\cG}<\mu_{\R}$ (for example, this is the case for the signed-post graph), then alternative ($ii$) holds, and we have existence of negative energy ground states for an interval of masses with positive lower bound. This is similar to what happens for the critical homogeneous problem, see \cite[Section 3]{AST2}.

%Theorem \ref{thm: main def tip} is focused on graphs with a terminal edge, i.e. of type (1) according to the classification in \cite{AST2}. In this case $\mu_{\cG}<\mu_{\R}$, but ground states never exist. 
%
%Theorem \ref{thm: main def no tip} includes three classes of graphs: graphs which admits a cycle covering - of type (2), graphs with only one half-line - of type (3), and graphs of type (4), i.e. graphs without terminal edges, which cannot be covered by cycle, and have at least two half-lines. Now, if $\cG$ is of type (2), then $\mu_\cG = \mu_{\R}$, so that ground states never exist; if $\cG$ is of type (3), then $\mu_{\cG}=\mu_{\R^+}$, and Theorem \ref{thm: main def no tip} applies to give existence of ground states for small values of $|\alpha|$; finally, there are graphs $\cG$ of type (4) such that $\mu_{\cG}<\mu_{\R}$, such as the signed post-graph. Also in this case, Theorem \ref{thm: main def no tip} applies. 
%
%Recalling the definition of $\tilde \mu_{\cG}$, it comes out that a necessary condition for the existence of ground states in the defocusing case is that $\mu_{\cG}<\tilde \mu_{\cG}$, and $\mu \in (\mu_{\cG}, \tilde \mu_{\cG}]$.
%
\end{remark}

Theorem \ref{thm: main def} reveals a further deep difference between the focusing and the defocusing cases. While in the former we may have existence (or non-existence) for every $\alpha>0$, only depending on $\mu$ and $\cG$, in the latter one there may be existence of ground states with a certain mass provided that $|\alpha|$ is sufficiently small, while there is non-existence for large $|\alpha|$. Moreover, it seems that the defocusing case is somehow less favorable for existence of ground states (for graphs of type (1) and (2), ground states never exist).

A final observation, again in this direction, regards the existence of \emph{critical points} (not necessarily ground states) for the NLS energy. In particular, in absence of ground states one may try to find \emph{local} minimizers of the energy, giving stable solutions for the time-dependent equation. In the homogeneous case $\alpha=0$, some results are contained in \cite{NP, PSV}. The problem with combined nonlinearities is unexplored, and will be the object of future investigation. For the moment, we limit ourselves to the following result concerning the star-graph. %(that is a graph made of $N \ge 3$ half-lines glued together at their common origin). %This class of graphs represents a prototypical non-trivial model for non-compact graphs with a cycle covering, and for this reason their study attracted a lot of attention in the last decade. 

\begin{proposition}\label{prop: non ex star}
Let $p \in (2,6)$, $\alpha<0$, and let $\cG_N$ denote the star graph consisting of $N \ge 3$ half-lines glued together at their common origin. If $0<\mu \le \mu_{\R}$, then there is no critical point of $E_\a(\cdot, \cG_N)$ on $H^1_\mu(\cG_N)$. For general $\mu >0$, there is no local minimizer of $E_\a(\cdot, \cG_N)$ on $H^1_\mu(\cG_N)$.
\end{proposition}

This is a further indication of the rigidity of the problem in the defocusing case $\alpha<0$.
%{\color{red} Da fare: sarebbe bello dimostrare che $\alpha_1=\alpha_2$, cio\`e avere una caratterizzazione sharp.}

\subsection*{Structure of the paper} In Section \ref{sec: pre} we collect several preliminary results which will be frequently used. In Section \ref{sec: foc} we present the proof of the main results with focusing perturbation ($\alpha>0$), while Section \ref{sec: def} concerns the defocusing case ($\alpha<0$).

%\subsection{Asymptotic of ground states with respect to $\alpha$}
%
%We analyze now the behavior of the ground states with respect to $\alpha$, in particular in the limit as $\alpha \to 0$. Given a non-compact graph $\cG$, we fix $\mu \in (0, \tilde \mu_{\cG})$, and suppose at first that $\cE_\a(\mu,\cG) \in (-\infty,0)$ is attained for every $\alpha>0$ sufficiently small. We denote by $u_\alpha$ the corresponding ground state, and try to understand the asymptotic behavior of $u_\alpha$ as $\alpha \to 0^+$. 

\section{Preliminaries}\label{sec: pre}

In this section we collect some preliminary results which will be frequently used throughout the rest of the paper.

\subsection{Gagliardo-Nirenberg inequality and critical mass}

First of all, we recall that for any $q >2$ and any non-compact graph $\cG$ the following Gagliardo-Nirenberg inequality holds (see e.g. \cite[Proposition 2.1]{AST3}): there exists an optimal constant $C_q(\cG)>0$ depending on $q$ and $\cG$ such that
\[
\lp{u}{q}^q \le  C_q(\cG) \lp{u}{2}^{\frac{q+2}2} \lp{u'}2^{\frac{q-2}2} \qquad \forall u \in H^1(\cG).
\]
Precisely, $C_q(\cG)$ is characterized as
\[
C_q(\cG) = \sup_{u \in H^1(\cG) \setminus \{0\}} \frac{\lp{u}{q}^q}{\lp{u}{2}^{\frac{q+2}2} \lp{u'}2^{\frac{q-2}2}}.
\]
A relevant role is played by the optimal constant obtained for the $L^2$-critical exponent $q=6$: indeed, the critical mass $\mu_{\cG}$ associated with $\cG$ is precisely equal to $\mu_{\cG}= \sqrt{ {3}/{C_6(\cG)} }$, see \cite[Section 2]{AST2}.

Dealing with a critical problem, it is also useful to recall the following \emph{modified Gagliardo-Nirenberg inequality}: %, which is a crucial ingredient in proving the existence of negative energy ground states for graphs of type 3 and 4.

\begin{lemma}[Lemma 4.4 in \cite{AST2}]\label{lem: mod GN}
Assume that $\cG$ is non-compact and has no terminal point, and let $u \in H^1_\mu(\cG)$ for some $\mu \in (0,\mu_\R]$. Then there exists $\theta_u \in [0,\mu]$ such that
\[
\lp{u}{6}^6\le 3\left(\frac{\mu-\theta_u}{\mu_{\R}}\right)^2 \lp{u'}{2}^2 + C_{\cG} \theta_u^\frac12,
\]
with $C_{\cG} >0$ depending only on $\cG$.
\end{lemma}

\subsection{Basic estimates on the ground state level}

Arguing as in \cite[Proposition 2.1]{AST3} or \cite[Proposition 2.3]{AST2}, it is not difficult to check that for any non-compact graph $\cG$ and any $q>2$
\beq\label{GN cost}
C_q(\R) \le C_q(\cG) \le C_q(\R^+) 
\eeq
(in this framework, $\R$ can be seen as a graph obtained by glueing together two copies of $\R^+$),  and that for every $\mu>0$ and $\alpha \in \R$
 \beq\label{dis gs}
 \cE_{\a}(\mu, \R^+) \le \cE_\alpha(\mu, \cG) \le \cE_\a(\mu, \R).
 \eeq
 In the particular case $q=6$, inequality \eqref{GN cost} yields
 \beq\label{crit mass}
 \frac{\mu_{\R}}{2} = \mu_{\R_+} \le \mu_\cG \le \mu_{\R}.
 \eeq
 
\begin{remark}\label{rmk: density}
The second inequality in \eqref{GN cost}, and the first one in \eqref{dis gs}, follow almost directly by considering decreasing rearrangements on $\R^+$. Instead, in order to prove the validity of the first inequality in \eqref{GN cost}, and of the second one in \eqref{dis gs}, one can use the following argument borrowed from \cite{AST2}, which will be also useful also in other parts of the paper. Since a non-compact graph $\cG$ contains at least one unbounded edge, it contains in turn arbitrarily large intervals. Therefore, any function $v \in H^1_{\mu,c}(\R):= \{u \in H^1_\mu(\R): \text{ $u$ has compact support}\}$ can be regarded as a function on $H^1_\mu(\cG)$, since we can place the support of $v$ on an half-line of $\cG$, and extend $v$ as $0$ on the rest of the graph. This argument will be summarized by saying that \emph{$H^1_\mu(\cG)$ contains $H^1_{\mu,c}(\R)$}, with some abuse of terminology. Now the first inequality in \eqref{GN cost}, and the second one in \eqref{dis gs}, simply follow from the fact that $H^1_{\mu,c}(\R)$ is dense in $H^1_\mu(\R)$ (see \cite[Proposition 2.3]{AST2} for more details).
\end{remark}

Estimates \eqref{GN cost}-\eqref{crit mass} make clear that the understanding of the cases $\cG = \R^+$ and $\cG =\R$ is essential. If $\alpha=0$, then 
\beq\label{gs 0 R}
\cE_0(\mu,\R) = \begin{cases} 0 & \text{if $\mu \in (0,\mu_{\R}]$} \\ -\infty & \text{if $\mu > \mu_{\R}$}, \end{cases} \qquad \cE_0(\mu,\R^+) = \begin{cases} 0 & \text{if $\mu \in (0,\mu_{\R^+}]$} \\ -\infty & \text{if $\mu > \mu_{\R^+}$}; \end{cases}
\eeq
moreover, $\cE_0(\mu,\R)$ (resp. $\cE_0(\mu,\R^+)$) is achieved if and only if $\mu = \mu_{\R}$ (resp. $\mu = \mu_{\R^+}$). Ground states for $\cE_0(\mu_{\R},\R)$, called solitons, form a two-parameters family $\phi_{\lambda,x_0}(x) = \sqrt{\lambda}\phi(\lambda(x-x_0))$, where $\lambda>0$, $x_0 \in \R$, and
\beq\label{soliton}
\phi(x) = \textrm{sech}^{1/2}\left( \frac{2x}{\sqrt{3}} \right).
\eeq
Ground states for $\cE_0(\mu_{\R^+},\R^+)$ are half-solitons, that is restrictions on $\R^+$ of even solitons.

The case $\alpha \neq 0$ and $\cG=\R$ is studied in \cite[Theorem 1.1]{So1}. With the above notations, the result reads as follows:

\begin{theorem}\label{thm: R}
Let $p \in (2,6)$. It results that:
\begin{itemize}
\item[($i$)] If $\alpha>0$, then 
\[
\cE_\alpha(\mu, \R) \in (-\infty,0) \quad \text{if $\mu \in (0, \mu_\R)$}, \quad \text{and}\quad \cE_\alpha(\mu, \R) =-\infty \quad \text{if $\mu  \ge \mu_\R$}.  
\]
Moreover, for $\mu \in (0, \mu_\R)$ the infimum is achieved by an even positive function, decreasing on $[0,+\infty)$.
\item[($ii$)] If $\alpha < 0$, then
\[
\cE_\alpha(\mu, \R) = 0 \quad \text{if $\mu \in (0, \mu_\R]$},  \quad \text{and}\quad \cE_\alpha(\mu, \R) =-\infty \quad \text{if $\mu >\mu_\R$}, 
\]
and the infimum is never achieved.
\end{itemize}
\end{theorem}

\begin{remark}
It is not difficult to prove that any ground state $\phi_{\mu, \alpha}$ for $\cE_\alpha(\mu, \cG)$ with $\alpha >0$ and $\mu \in (0,\mu_{\R})$ must be a positive (or negative) function, even with respect to a point $x_0$, and decreasing (or increasing) in $(x_0,+\infty)$. This follows from the maximum principle (using the fact that any ground state solves the NLS equation) and a simple rearrangement argument (for which one can use e.g. \cite[Proposition 3.1]{AST1}).  

Moreover, proceeding as in \cite{So1}, one can prove a result completely analogue to Theorem \ref{thm: R} for $\cG=\R^+$, with $\mu_\R$ replaced by $\mu_{\R^+} = \mu_{\R}/2$. %Ground states for $\cE_0(\mu_{\R^+},\R^+)$ are \emph{half-solitons}, that is restrictions on $\R^+$ on the even solitons $\phi_{\lambda,0}$.

For future convenience, we compare the ground state levels obtained on $\R$ and on $\R^+$. For $\mu \in (0,\mu_{\R^+})$ and $\alpha>0$, by a rearrangement argument
\[
\cE_\alpha(2\mu,\R) = \inf\left\{E_\alpha(u,\R): \ \text{$u \in H^1_{2\mu}(\R)$ is even}\right\}.
\]
But any even function in $H^1_{2\mu}(\R)$ can be obviously identified with an element of $H^1_\mu(\R^+)$, and the identification is $1-1$. Therefore
\begin{equation}\label{gs R e R+}
\cE_\alpha(2\mu,\R) = 2\cE_\alpha(\mu,\R^+),
\eeq
and $\phi_{2\mu,\alpha}$ is an even ground state for $\cE_\a(2\mu,\R)$ if and only if $\phi_{2\mu,\alpha}|_{\R^+}$ is a even ground state for $\cE_\alpha(\mu, \R^+)$.
%In particular, for $\cG=\R^+$ and $\cG=\R$ ground states never exist in the defocusing case $\alpha<0$, and exist in the focusing case $\alpha > 0$ only for masses smaller than the critical mass. 
\end{remark}

Let us now come back to general graphs $\cG$. By the Gagliardo-Nirenberg inequality and the definition of $\mu_{\cG}$, for every $u \in H^1_\mu(\cG)$
\begin{equation}\label{E da sotto}
E_\alpha(u, \cG) \ge \frac12\left(1-\left(\frac{\mu}{\mu_{\cG}}\right)^2\right) \lp{u'}{2}^2 - \frac{\alpha}{p} C_p(\cG) \mu^\frac{p+2}{4} \lp{u'}{2}^{\frac{p-2}2}.
\eeq
Since $p<6$, it follows that $\cE_\alpha(\mu, \cG)$ is bounded from below for every $\mu \in (0,\mu_{\cG})$, for every $\alpha \in \R$. 

Collecting together what we recalled so far, we obtain a preliminary result.

%If we assume that $\alpha  \le 0$, we can easily say something more.

\begin{proposition}\label{prop: basic}
Let $\cG$ be a non-compact graph, and let $p \in (2,6)$.
\begin{itemize}
\item[($i$)] If $\alpha>0$, then 
\[
\cE_\alpha(\mu, \cG) \in (-\infty,0) \quad \text{if $\mu \in (0, \mu_\cG)$}, \quad \text{and}\quad \cE_\alpha(\mu, \cG) =-\infty \quad \text{if $\mu  \ge \mu_\R$}.  
\]
Moreover, for $\mu \in [\mu_{\cG}, \mu_{\R})$ we have that $\cE_\alpha(\mu, \cG) <0$ (possibly $-\infty$).
\item[($ii$)] If $\alpha < 0$, then 
\[
\cE_\alpha(\mu, \cG) = 0 \quad \text{if $\mu \in (0, \mu_\cG]$}, \quad \text{and}\quad \cE_\alpha(\mu, \cG) =-\infty \quad \text{if $\mu > \mu_\R$}
\]
Moreover, the infimum is never achieved when $\mu \in (0,\mu_{\cG}]$ and $\alpha<0$.
\end{itemize}
\end{proposition}
\begin{proof}
Point ($i$) follows directly from \eqref{dis gs}, Theorem \ref{thm: R} and estimate \eqref{E da sotto}. 
%The fact that $\cE_\alpha(\mu, \cG) <0$ for $\mu \in [\mu_{\cG}, \mu_{\R})$ can be checked by considering the elements of a maximizing sequence in $H^1_\mu(\cG)$ for $C_6(\cG) = 3/\mu_{\cG}^2$ as test functions of $E_\a(\cdot\,,\cG)$.

Concerning point ($ii$), we observe that for any $\alpha<0$
\[
\cE_\a(\mu,\cG) \le \cE_\a(\mu,\R) = \begin{cases} 0 & \text{if $\mu \in (0,\mu_{\R}]$} \\ -\infty & \text{if $\mu > \mu_{\R}$} \end{cases}
\]
by \eqref{dis gs} and Theorem \ref{thm: R}. Moreover, if $\mu \in (0,\mu_{\cG}]$ estimate \eqref{E da sotto} gives
\[
E_\alpha(u, \cG) > 0 \qquad \forall u \in H^1_\mu(\cG), \quad \forall \mu \in (0,\mu_{\cG}],
\]
so that $\cE_\a(\mu,\cG) \ge 0$ for any such $\mu$. We infer that $\cE_\a(\mu, \cG) = 0$ for $\mu \in (0,\mu_{\cG}]$, and the infimum is never achieved.
%
%
%On the other hand
%\[
%0 \le \cE_\a(\mu,\cG) \le \cE_\a(\mu,\R) = 0 \qquad \forall \mu \in (0,\mu_{\cG}],
%\]
%where we used \eqref{dis gs} and Theorem \ref{thm: R}. Thus, $\cE_\a(\mu, \cG) = 0$ for $\mu \in (0,\mu_{\cG}]$, and the infimum cannot be achieved. If instead $\mu >\mu_{\R}$, then $\cE_\alpha(\mu, \cG) = -\infty$ follows from \eqref{dis gs} and Theorem \ref{thm: R}.
%
%Suppose now that $\cE_\alpha(\mu,\cG)$ is achieved by $u \in H^1_\mu(\cG)$, for some $\alpha<0$ and $\mu \in (0,\mu_{\cG}]$. Then 
%\[
%0 = \cE_0(\mu,\cG)  \le E_0(u, \cG) < E_\a(u, \cG) = \cE_\a(\mu,\cG) = 0,
%\]
%a contradiction.
\end{proof}

This means that the interesting cases are:
\begin{itemize}
\item[($i$)] $\alpha>0$ and $\mu \in (0, \mu_{\R})$;
\item[($ii$)] $\alpha<0$ and $\mu \in (\mu_{\cG}, \mu_{\R}]$
\end{itemize}
We shall see that estimates on $\cE_\a(\mu, \cG)$ and the existence of ground states in these ranges depend strongly on the properties of $\cG$. %For instance, if $\mu_{\cG} = \mu_{\R}$ there are never ground states for $\alpha<0$. 
In this perspective, it is convenient to review the results proved in \cite{AST2} for the homogeneous case $\alpha = 0$.

\section{Proof of the main results  in the focusing case $\alpha>0$}\label{sec: foc}

Throughout this section we consider the case when $\alpha>0$. 

\subsection{Proof of Theorem \ref{thm: main foc}.} The proof requires some preliminary statements. At first, we characterize the ground state energy level.

\begin{lemma}\label{lem: level}
%Let $\cG$ be a non-compact metric graph, $p \in (2,6)$, $\alpha>0$, and let $\tilde \mu_{\cG}$ be defined by \eqref{new crit}. 
We have that 
\[
\cE_\alpha(\mu, \cG)  \in (-\infty,0) \quad \text{if $\mu \in (0,\tilde \mu_{\cG})$}, \quad \text{and} \quad \cE_\alpha(\mu, \cG)  = -\infty \quad \text{if $\mu \ge \tilde \mu_{\cG}$}.
\]
\end{lemma}

\begin{proof}
By Proposition \ref{prop: basic}, we know that $\cE_\alpha(\mu, \cG) \in (-\infty,0)$ if $\mu \in (0, \mu_{\R^+})$, and $\cE_\a(\mu, \cG) = -\infty$ if $\mu \ge \mu_{\R}$, for any non-compact $\cG$ and $\alpha>0$. 

Let $\cG$ be a graph without a terminal point, so that $\tilde \mu_{\cG} = \mu_{\R}$, and let $\mu < \mu_{\R}$. By Lemma \ref{lem: mod GN}, there exists $C>0$ (independent of $u$) such that
\[
\begin{split}
E_\a(u,\cG) 
%&\ge \frac12\left(1-\left(\frac{\mu-\theta}{\mu_{\R}}\right)^2\right)\lp{u'}{2}^2- \frac{C}{6}\theta^\frac12- \frac{\alpha C_p(\cG)}{p} \mu^\frac{p+2}{4} \lp{u'}{2}^\frac{p-2}{4} \\
 \ge \frac12\left(1-\left(\frac{\mu}{\mu_{\R}}\right)^2\right)\lp{u'}{2}^2 - \frac{C_{\cG}}{6}\mu^\frac12 - \frac{\alpha C_p(\cG)}{p} \mu^\frac{p+2}{4} \lp{u'}{2}^\frac{p-2}{2},
\end{split}
\]
for every $u \in H^1_\mu(\cG)$. Since $p<6$, it follows plainly that $\cE_\a(\mu, \cG)>-\infty$. Moreover, by \eqref{dis gs} and Theorem \ref{thm: R} we also have that $\cE_\a(\mu, \cG)<0$. This completes the proof for graphs without terminal points. 

Let now $\cG$ be a graph with a terminal point, so that $\tilde \mu_{\cG} = \mu_{\R^+}$. In this case, we have to show that $\cE_\a(\mu,\cG) = -\infty$ for $\mu \in [\mu_{\R^+},\mu_{\R})$. If $\mu \in (\mu_{\R^+},\mu_{\R})$, it is sufficient to observe that, since $\alpha>0$, it results that $E_\alpha(u, \cG) \le E_0(u, \cG)$ for every $u \in H^1_\mu(\cG)$, and hence $\cE_\a(\mu,\cG) \le \cE_0(\mu,\cG) =-\infty$, where the last equality is proved in \cite[Theorem 3.1]{AST2}. We focus now on $\mu = \mu_{\R^+}$. We denote by $\phi$ the restriction on $\R^+$ of the standard soliton defined in \eqref{soliton}, and by $\phi_\lambda(x) = \sqrt{\lambda} \phi(\lambda x)$. Let $\textrm{e}$ be a terminal edge of $\cG$, i.e. an edge with a terminal point. We identify $\textrm{e}$ with $[0,\ell]$, with the coordinate $0$ placed at the terminal point, and define $u_\lambda \in H^1_{\mu_{\R^+}}(\cG)$ by
\[
u_\lambda(x):= \begin{cases}
\frac{\mu_{\R^+}^{1/2}}{\| \phi_\lambda-\phi_\lambda(\ell)\|_{L^2(0,\ell)}} (\phi_\lambda(x)-\phi_\lambda(\ell)) & \text{if $x \in \textrm{e}$} \\
0 & \text{if $x \in \cG \setminus \textrm{e}$.} 
\end{cases}
\]
It is easy to check that $\mu_{\R^+} > \| \phi_\lambda-\phi_\lambda(\ell)\|^2_{L^2(0,\ell)} \to \mu_{\R^+}$ as $\lambda \to +\infty$. Therefore, if $1+\delta_\lambda>1$ denotes the ratio $\mu_{\R^+}^{1/2}/\|\phi_\lambda-\phi_\lambda(\ell)\|_{L^2(0,\ell)}$, we have that
\beq\label{04061}
\begin{split}
E_\a(u_\lambda,\cG) &= E_\a(u_\lambda,\textrm{e}) \\
&= \frac{(1+\delta_\lambda)^2}{2} \int_0^\ell |\phi_\lambda'|^2 - \frac{(1+\delta_\lambda)^6}{6} \int_0^\ell (\phi_\lambda-\phi_\lambda(\ell))^6 - \alpha\frac{(1+\delta_\lambda)^p}{p} \int_0^\ell (\phi_\lambda-\phi_\lambda(\ell))^p \\
& \le (1+\delta_\lambda)^2\left[\frac{1}{2} \int_0^\ell |\phi_\lambda'|^2 - \frac{1}{6} \int_0^\ell (\phi_\lambda-\phi_\lambda(\ell))^6 - \frac{\alpha}{p} \int_0^\ell (\phi_\lambda-\phi_\lambda(\ell))^p \right] \\
& = (1+\delta_\lambda)^2\left[ \lambda^2 E_0(\phi,(0,\lambda \ell)) - \frac{\alpha \lambda^{\frac{p-2}2}}{p} \int_0^{\lambda \ell} \phi^p  \right. \\
& \hphantom{= (1+\delta_\lambda)^2\bigg[ } \quad +\left. \frac{\lambda^2}6 \int_0^{\lambda \ell} \left[\phi^6 - \left(\phi-\phi(\lambda\ell)\right)^6\right] + \frac{\alpha \lambda^{\frac{p-2}{2}}}{p} \int_0^{\lambda \ell} \Big[ \phi^p - \left(\phi-\phi(\lambda\ell)\right)^p\Big]\right]
\end{split}
\eeq
Now we estimate separately each term on the right hand side. Firstly, recalling that the half-soliton has null energy on $\R^+$, we observe that
\[
E_0(\phi,(0,\lambda \ell)) = E_0(\phi,\R^+) - E_0(\phi,(\lambda \ell, +\infty)) = - E_0(\phi(\cdot + \lambda \ell),\R^+) <0,
\]
since the mass of the translated half-soliton is smaller than $\mu_{\R^+}$. Secondly, we have that for any $q  > 2$
\[
0 \le \int_0^{\lambda \ell} \Big[ \phi^q - \left(\phi-\phi(\lambda\ell)\right)^q\Big] \le q \phi(\lambda \ell) \int_0^{\lambda \ell} \phi^{q-1} \le C \phi(\lambda \ell),
\]
where we used the integrability of $\phi^{q-1}$ on $[0,+\infty)$, which is a direct consequence of the exponential decay of $\phi$. This also implies that $\lambda^q \phi(\lambda\ell) \le C \lambda^q e^{-C \lambda \ell} \to 0$ as $\lambda \to +\infty$. Therefore, coming back to \eqref{04061}, we have that for $\lambda$ sufficiently large
\[
\begin{split}
E_\a(u_\lambda,\cG) &\le (1+\delta_\lambda)^2 \left[ - \frac{\alpha \lambda^{\frac{p-2}2}}{p} \int_0^{\lambda \ell} \phi^p + C_1 \lambda^2 \phi(\lambda \ell) + C_2 \alpha \lambda^{\frac{p-2}2}\phi(\lambda \ell) \right] \\
& \le -  \frac{\alpha \lambda^{\frac{p-2}2}}{2p} \int_0^{+\infty} \phi^p + C \lambda^2 e^{-C \lambda \ell} \to -\infty
\end{split}
\]
as $\lambda \to +\infty$.
\end{proof}

\begin{lemma}\label{lem: cont}
The map $\cE_\a(\cdot, \cG): [0,\tilde{\mu}_{\cG}) \to \R$, extended as $0$ in $\mu=0$, is continuous.
%\[
%\mu \in [0, \tilde \mu_{\cG}) \mapsto \begin{cases} \cE_\a(\mu,\cG) & \text{if $\mu \in (0,\tilde \mu_{\cG})$}  \\ 0 & \text{if $\mu=0$} \end{cases}
%\]
%is continuous.
\end{lemma}

\begin{proof}
Let $\mu \in (0,\tilde \mu_{\cG})$, $\mu_n \to \mu$, and $\eps>0$ such that $\mu + \eps<\tilde \mu_{\cG}$. For every $n$ large, we have that $\mu_n <\mu+\eps$ and there exists $u_n \in H^1_{\mu_n}(\cG)$ such that 
\[
\cE_\a(\mu_n,\cG) \le E_\a(u_n, \cG) \le \cE_\alpha(\mu_n, \cG) + \frac1n \le 1,
\]
where we used the fact that $\cE_\alpha(\mu_n, \cG)<0$. If $\cG$ has no terminal point, then this estimate and the modified Gagliardo-Nirenberg inequality (Lemma \ref{lem: mod GN}) imply that
\[
\frac12\left(1-\left(\frac{\mu+\eps}{\mu_{\R}}\right)^2\right)\lp{u_n'}2^2 - C_{\cG} (\mu+\eps)^\frac12- \frac{\alpha C_p(\cG)}p (\mu+\eps)^\frac{p+2}{4}\lp{u_n'}{2}^{\frac{p-2}{2}} \le E_\a(u_n,\cG) \le 1.
\] 
Thanks to the choice of $\eps$, the coefficient of $\lp{u_n'}2^2$ is positive, and, since $p<6$, we infer that $\{u_n\}$ is bounded in $H^1(\cG)$. The same argument also works when $\cG$ has a terminal point, by simply using the standard Gagliardo-Nirenberg inequality instead of the modified one. Thus, in both cases we consider $v_n = \mu^{1/2} u_n/\mu_n^{1/2} \in H^1_\mu(\cG)$, and notice that
\[
\begin{split}
\cE_\a(\mu,\cG) &\le E_\a(v_n, \cG) \\
&  = E_\a(u_n, \cG) + \frac12\left(\frac{\mu}{\mu_n}-1\right)\lp{u_n'}{2}^2 \\
& \hphantom{=E_\a(u_n, \cG) }  \ - \frac16\left( \left(\frac{\mu}{\mu_n}\right)^3-1\right)\lp{u_n}{6}^6 - \frac{\alpha}{p} \left( \left(\frac{\mu}{\mu_n}\right)^{\frac{p}2}-1\right)\lp{u_n}{p}^p \\
& = E_\alpha(u_n, \cG) + o(1) \le \cE_\a(\mu_n, \cG) + o(1),
\end{split}
\]
 as $n \to \infty$, where the last equality follows from the boundedness of $\{u_n\}$ in $H^1(\cG)$ (and hence also in $L^6$ and $L^p$). We deduce that
 \begin{equation}\label{low cont}
 \cE_\a(\mu,\cG) \le \liminf_{n \to \infty} \cE_\a(\mu_n,\cG).
 \eeq
 On the other hand, let $\{w_n\} \subset H^1_\mu(\cG)$ be a minimizing sequence for $\cE_\a(\mu,\cG)$. By using the same argument based on the Gagliardo-Nirenberg inequality as above, we have that $\{w_n\}$ is bounded in $H^1(\cG)$, and hence, by letting $z_n = \mu_n^{1/2} w_n/\mu^{1/2} \in H^1_{\mu_n}(\cG)$, we have that
 \[
\cE_\a(\mu_n,\cG) \le E_\a(z_n, \cG) = E_\a(w_n, \cG) + o(1) = \cE_\a(\mu, \cG) + o(1),
\]
whence
\[
\limsup_{n \to \infty} \cE_\a(\mu_n,\cG) \le \cE_\a(\mu,\cG).
\]
This and \eqref{low cont} prove the continuity of $\cE_\a(\cdot,\cG)$ at any $\mu \in (0,\tilde \mu_{\cG})$. 

Let now $\mu_n \to 0^+$. We aim to show that $\cE_\a(\mu_n, \cG) \to 0$. We consider again
$u_n \in H^1_{\mu_n}(\cG)$ such that 
\[
\cE_\a(\mu_n,\cG) \le E_\a(u_n, \cG) \le \cE_\alpha(\mu_n, \cG) + \frac1n \le \frac1n.
\] 
The (standard) Gagliardo-Nirenberg inequality yields
\[
\frac12\left(1-\left(\frac{\mu_n}{\mu_{\R}}\right)^2\right)\lp{u_n'}2^2 - \frac{\alpha C_p(\cG)}p (\mu_n)^\frac{p+2}{4}\lp{u_n'}{2}^{\frac{p-2}{2}} \le \frac1n.
\] 
We deduce firstly that $\{u_n\}$ is bounded in $H^1$, and, afterwards, that $\lp{u_n'}2^2 \to 0$, which in turn implies that $\cE_\a(\mu_n, \cG) = E_\a(u_n,\cG) + o(1) \to 0$. 
\end{proof}

Next we show that $\cE_\alpha(\cdot, \cG)$ is strictly subadditive.

\begin{lemma}\label{lem: sub}
Let $\mu_1,\mu_2>0$ be such that $\mu_1+\mu_2 < \tilde \mu_{\cG}$. Then
\[
\cE_\a(\mu_1+\mu_2, \cG) < \cE_\a(\mu_1, \cG) + \cE_\a(\mu_2, \cG).
\]
In addition, the map $\cE_\alpha(\mu, \cG)$ is strictly decreasing in $[0,\tilde \mu_{\cG})$.
\end{lemma} 

\begin{proof}
Let $\mu \in (0,\tilde \mu_{\cG})$ and $\theta>1$ be such that $\theta \mu <\tilde \mu_{\cG}$. If $\{u_n\} \subset H^1_\mu(\cG)$ is a minimizing sequence for $\cE_\a(\mu, \cG)$, then $\theta^{1/2} u_n \in H^1_{\theta \mu}(\cG)$, and 
\[
\begin{split}
\cE_\a(\theta \mu, \cG) & \le \frac{\theta}2 \lp{u_n'}{2}^2 - \frac{\theta^3}6\lp{u_n}{6}^6 - \frac{\alpha \theta^{p/2}}{p}\lp{u_n}{p}^p < \theta E_\a(u_n,\cG),
\end{split}
\]
since $\theta>1$, $p>2$, and $\alpha>0$. It follows plainly that $\cE_\a(\theta \mu, \cG) \le \theta \cE_\alpha(\mu, \cG)$, with equality only if $\lp{u_n}{p}^p + \lp{u_n}{6}^6 \to 0$. This is however not possibile, since, if it were true, it would imply that
\[
0> \cE_\a(\mu, \cG) = \lim_{n} E_\alpha(u_n, \cG) \ge \liminf_n \frac12 \lp{u_n'}{2}^2 \ge 0,
\]
a contradiction. Therefore
\beq\label{dis sub}
\cE_\a(\theta \mu, \cG) < \theta \cE_\alpha(\mu, \cG),
\eeq
which implies that $\cE_\alpha(\mu, \cG)$ is strictly decreasing in $\mu \in [0,\tilde \mu_{\cG})$. Moreover, from \eqref{dis sub} we also deduce the subadditivity: if $0 <\mu_1 \le \mu_2$ with $\mu_1+\mu_2 =\mu <\tilde \mu_{\cG}$, then by using \eqref{dis sub} twice we infer that
\[
 \cE_\a(\mu, \cG) < \frac{\mu}{\mu_2} \cE_\a(\mu_2, \cG) = \cE_\a(\mu_2, \cG) + \frac{\mu_1}{\mu_2} \cE_\a(\mu_2, \cG) \le \cE_\a(\mu_2, \cG) +  \cE_\a(\mu_1, \cG),
 \]
 which is the desired result.
\end{proof}

Finally:

\begin{lemma}\label{lem: inf <=}
For $\mu \in (0,\tilde \mu_{\cG})$, let 
\[
\tilde \cE_\a(\mu,\cG) := \inf\left\{ E_\alpha(u, \cG): \ u \in H^1_\mu(\cG) \ \text{and} \  \int_{\cG} |u|^2 \le \mu \right\}.
\]
Then $\tilde \cE_\a(\mu,\cG) = \cE_\a(\mu,\cG)$.
\end{lemma}

\begin{proof}
Proceeding as in Lemma \ref{lem: level}, we have that $\tilde \cE_\a(\mu,\cG) \in (-\infty,0)$. Let $\{u_n\} \subset H^1(\cG)$ be a minimizing sequence for $\tilde \cE_\a(\mu,\cG)$, and let $\mu_n = \lp{u_n}{2}^2$ be the mass of $u_n$. Clearly $\mu_n \to \nu \in [0,\mu]$ (up to a subsequence) and, by using the Gagliardo-Nirenberg inequality (the standard one if $\cG$ has a terminal point, the modified one if not), it is easy to check that $\{u_n\}$ is bounded in $H^1(\cG)$. Then, arguing as in Lemma \ref{lem: cont}, it is possible to see that the sequence $v_n:=\nu^{1/2} u_n/\mu_n^{1/2}$ is a new minimizing sequence: $E_\a(v_n,\cG) \to \tilde \cE_\a(\mu,\cG)$. Since all the functions $v_n$ have same mass equal to $\nu \le \mu$, we have that 
\[
\tilde \cE_\a(\mu, \cG) \le \cE_\a(\nu,\cG) \le E_\a(v_n, \cG) \to \tilde \cE_\a(\mu, \cG),
\]
whence it follows that $\tilde \cE_\a(\mu, \cG) = \cE_\a(\nu,\cG)$. Now, if $\nu = \mu$, then the proof is complete. If instead $\nu<\mu$, recalling that $\cE_\a(\cdot, \cG)$ is strictly decreasing, we obtain $\tilde \cE_\a(\mu, \cG) = \cE_\a(\nu,\cG) >  \cE_\a(\mu,\cG)$, which is a contradiction since by definition $\tilde \cE_\a(\mu, \cG) \le \cE_\a(\mu, \cG)$.
\end{proof}

\begin{lemma}\label{lem: compactness}
For $\mu \in (0,\tilde \mu_{\cG})$, any minimizing sequence $\{u_n\} \subset H^1_\mu(\cG)$ for $\cE_\alpha(\mu, \cG)$ is weakly compact in $H^1(\cG)$. If $u_n \weak u$ weakly in $H^1(\cG)$, then one of the following alternative occurs:
\begin{itemize}
\item[($i$)] either $u_n \to 0$ in $L^\infty_{\loc}(\cG)$, and $u \equiv 0$; 
\item[($ii$)] or $u \in H^1_\mu(\cG)$, $u_n \to u$ strongly in $H^1(\cG) \cap L^6(\cG)$, and $u$ is a ground state for $\cE_\a(\mu, \cG)$.
\end{itemize}
\end{lemma}

\begin{proof}
Thanks to Lemmas \ref{lem: cont}-\ref{lem: sub}, we can adapt the proof of \cite[Theorem 3.2]{AST3}. We report a sketch for the sake of completeness. 

The boundedness of $\{u_n\}$ follows by the Gagliardo-Nirenberg inequality (the standard one if $\cG$ has a terminal point, the modified one if not). Therefore, up to a subsequence $u_n \weak u$ weakly in $H^1(\cG)$, locally uniformly on $\cG$, and almost everywhere. In particular,
\[
m:= \int_{\cG} u^2 \le \liminf_n \int_{\cG} u_n^2 = \mu.
\]
Now, by using the Brezis-Lieb lemma and the weak convergence, it is not difficult to check that
\[
E_\a(u_n , \cG) = E_\a(u, \cG) + E_\a(u_n-u, \cG).
\]
Denoting by $\nu_n = \int_{\cG} |u_n-u|^2$, we have that $\nu_n \to \mu-m$ and, by Lemma \ref{lem: cont}, we deduce that 
\beq\label{above ineq}
\begin{split}
\cE_\a(\mu, \cG) + o(1) &= E_\a(u_n , \cG) = E_\a(u, \cG) + E_\a(u_n-u, \cG) \\
& \ge E_\a(u, \cG) + \cE_\a(\mu-m, \cG) + o(1) \\
& \ge  \cE_\a(m, \cG) + \cE_\a(\mu-m, \cG) + o(1)
\end{split}
\eeq
as $n \to \infty$, whence it follows that $\cE_\a(\mu, \cG) \ge \cE_\a(m, \cG) + \cE_\a(\mu-m)$. If $0<m<\mu$, this is in contradiction with the strict subadditivity of $\cE_\a(\cdot\,,\cG)$ in $(0,\tilde \mu_{\cG})$, Lemma \ref{lem: sub}, and hence either $m=0$, or $m=\mu$. If $m=0$, then alternative ($i$) of the thesis holds; if instead $m=\mu$, then the \eqref{above ineq} and the fact that $\cE_\a(0,\cG) = 0$ (see Lemma \ref{lem: cont}) imply that
\[
\cE_\a(\mu, \cG) \ge E_\a(u, \cG) + \cE_\a(0,\cG) = E_\a(u, \cG),
\]
so that $u \in H^1_\mu(\cG)$ and is the desired ground state. Moreover, since $m=\mu$, $u_n \to u$ strongly in $L^2(\cG)$; this and the boundedness of $\{u_n\}$ in $H^1$ imply that $u_n \to u$ strongly in $L^p(\cG)$ and in $L^6(\cG)$. Then, using that $E_\a(u, \cG) = \lim_n E_\a(u_n, \cG)$, the strong convergence in $H^1(\cG)$ follows.
\end{proof}

We are finally ready for the:
\begin{proof}[Proof of Theorems \ref{thm: main foc} and Corollary \ref{cor: foc}]
Thanks to Lemmas \ref{lem: cont}-\ref{lem: compactness}, we can adapt the proofs of Theorems 3.3 and Corollary 3.4 in \cite{AST3}. Again, we sketch the proof here.

Let $\{u_n\}$ be a minimizing sequence for $\cE_\a(\mu, \cG)$, with $\mu \in (0,\tilde \mu)$. Clearly, we can suppose that $u_n \ge 0$ on $\cG$, for every $n$. To prove Theorem \ref{thm: main foc}, it is sufficient to rule out alternative ($i$) in Lemma \ref{lem: compactness}. Let $\eps_n$ be the maximum of $u_n$ on the compact core of $\cG$ (that is, the subgraph of $\cG$ obtained by removing from $\cG$ all the open unbounded edges). If, by contradiction, $u_n \to 0$ locally uniformly on $\cG$, then $\eps_n \to 0$. If moreover $\|u_n\|_{L^\infty(\cG)} = \eps_n$, then clearly also $\lp{u_n}{p}^p + \lp{u_n}{6}^6 \to 0$, and hence \[
\cE_\a(\mu, \cG) = \lim_n E_\a(u_n, \cG) \ge \liminf_n \frac12 \lp{u_n'}{2}^2 \ge 0,
\]
in contradiction with the fact that $\cE_\a(\mu, \cG)<0$. Therefore, $\|u_n\|_{L^\infty(\cG)} > \eps_n$, and the function $v_n := \max\{0, u_n-\eps_n\}$, which is equal to $0$ on the compact core of $\cG$, does not vanish identically. On the other hand, on each half-line of $\cG$ we know that $u_n \to 0$, and hence, for every $t \in (0,\max v_n)$ the number of preimages $v_n^{-1}(t)$ is at least $2$. Clearly $\lp{v_n}{2} \le \mu$; thus, by considering the symmetric rearrangement $\hat v_n$ of $v_n$ on $\R$, by \cite[Proposition 3.1]{AST1} we infer that
\beq\label{riarr}
E_\a(v_n, \cG) \ge E_\a(\hat v_n, \R) \ge \tilde \cE_\a(\mu, \R) = \cE_\a(\mu, \R),
\eeq
where the last inequality follows by Lemma \ref{lem: inf <=}. On the other hand, since 
\[
\int_{\cG} |u_n - v_n|^q \le \lp{u_n-v_n}{\infty}^{q-2} \int_{\cG} |u_n-v_n|^2 \le \eps_n^{q-2} \mu \to 0
\]  
for every $q >2$, the sequence $\{v_n\}$ is still a minimizing sequence for $\tilde \cE_\a(\mu, \cG) = \cE_\a(\mu, \cG) < \cE_\a(\mu, \R)$, in contradiction with \eqref{riarr}. This completes the proof of the theorem. Corollary \ref{cor: foc} follows straightforwardly.
\end{proof}

%\begin{remark}
%The proof of Lemma \ref{lem: inf <=} exploits the validity of Lemmas \ref{lem: cont} and \ref{lem: sub}, while the fact that $\alpha>0$ is never used. This means that if we can show that Lemmas \ref{lem: cont} and \ref{lem: sub} hold for some $\alpha<0$, then also Lemma \ref{lem: inf <=} holds for such values. The same discussion applies to Lemma \ref{lem: compactness}, which relies on Lemmas \ref{lem: cont}-\ref{lem: inf <=}.
%\end{remark}

\begin{remark}
Although the conclusion of the proof of Theorem \ref{thm: main foc} is borrowed from \cite{AST3}, the key ingredients (continuity, monotonicity, subadditivity of the ground state energy level) cannot be directly adapted from \cite{AST3}. Indeed, the arguments in \cite{AST3} rely on nice scaling properties of the homogeneous NLS, which make the ground state energy level strictly concave with respect to $\mu \in (0,+\infty)$. These properties are destroyed when dealing with a combined inhomogeneous nonlinearity, and, in turn, $\mu \mapsto \cE_\a(\mu,\cG)$ is no more globally strictly concave. We overcame this difficulty by directly proving monotonicity and subadditivity of $\mu \mapsto \cE_\a(\mu,\cG)$ in the interval $(0,\tilde \mu_{\cG})$. %The failure of the scis what induces the appearance of the threshold value $\tilde \mu_{\cG}$ for the mass, and 
\end{remark}

%We focus now on more specific classes of graphs.

\subsection{Graphs which admit a cycle covering.} 

\begin{proof}[Proof of Theorem \ref{thm: cycle}]
The proof is a straightforward adaptation of \cite[Theorem 2.5]{AST1}. As in the application of Theorem \ref{thm: main foc}, the adaptation from \cite{AST1} is possible since ground states for $\cE_\a(\mu, \R)$ have the same property of the standard soliton for the homogeneous problem (it is even with respect to a point $x_0$, and strictly decreasing from that point on).
\end{proof}

\subsection{Graphs with a terminal edge.} 

In this subsection we prove Propositions \ref{prop: ex tip} and \ref{prop: non ex tip}. In proving the former result, we shall use the following preliminary lemma.

\begin{lemma}\label{lem: gs R+ R}
For every $\mu \in (0, \mu_{\R^+})$ and $\alpha > 0$, we have that
\[
\cE_\a(\mu, \R^+) < \cE_\a(\mu, \R).
\]
\end{lemma}

\begin{proof}
Let $\phi_{\mu,\a}$ be a ground state for $\cE_\a(\mu, \R)$. Since $\alpha>0$ and $p>2$, we have that
\[
\begin{split}
E_\a(\sqrt{2} \phi_{\mu,\a}, \R)  & = 2 \left[ \frac12 \| \phi_{\mu, \a}'\|_{L^2(\R)}^2  -\frac46 \| \phi_{\mu, \a}\|_{L^6(\R)}^6- \frac{\alpha}{p}2^\frac{p-2}2 \| \phi_{\mu, \a}\|_{L^p(\R)}^p \right] \\
& <2 E_\a(\phi_{\mu,\a}, \R) = 2 \cE_\a(\mu,\R).
\end{split}
\] 
Then, recalling \eqref{gs R e R+},
\[
\cE_\a(\mu,\R^+) = \frac12  \cE_\a(2\mu,\R) \le \frac12 E_\a(\sqrt{2} \phi_{\mu,\a}, \R) <  \cE_\a(\mu,\R). \qedhere
\]
\end{proof}

\begin{proof}[Proof of Proposition \ref{prop: ex tip}]
Let $\cG$ be a non-compact metric graph with a terminal edge of length $\ell$. For fixed $\alpha>0$ and $\mu \in (0,\mu_{\R^+})$, we aim to show that $\cE_\a(\mu,\cG)$ is attained if $\ell$ is larger than a constant $\bar \ell$ depending on $\alpha$ and $\mu$. By Lemma \ref{lem: gs R+ R}, we know that $\cE_\a(\mu, \R^+) < \cE_\a(\mu, \R)$. Therefore, by density, there exists $u \in H^1_\mu(\R^+)$ with compact support $[0,M]$ such that $E_\a(u, \R^+) < \cE_\a(\mu, \R)$. If $\ell \ge M$, we can see $u$ as a function in $H^1_\mu(\cG)$ (still denoted by $u$), with support contained on the terminal edge, extended as $0$ elsewhere. For this function $E_\a(u, \cG) < \cE_\a(\mu, \R)$, and hence Theorem \ref{thm: main foc} directly implies the thesis.
\end{proof}

%Lemma \ref{lem: gs R+ R} allowed us to adapt the proof of \cite[Proposition 4.1]{AST3} in the present setting. In \cite{AST3}, the authors deal with a subcritical homogeneous equation; in such case the estimate in Lemma \ref{lem: gs R+ R} holds for every $\mu$, and can be proved simply by scaling. Instead, dealing with a combined nonlinearity with focusing leading term, the estimate does not hold for $\mu \ge \mu_{\R^+}$, since the left hand side is $-\infty$.

%\medskip

We now turn to the proof of Proposition \ref{prop: non ex tip}. We shall use a preliminary result which is a partial counterpart of \cite[Theorem 4.3]{AST3}. Let $\cG$ be a non-compact graph and, for $n \in \N$, let $\mathcal{K}_n$ be a connected compact graph, of total length $|\mathcal{K}_n|$. We denote by $\cG_n$ the graph obtained by attaching $\mathcal{K}_n$ to $\cG$ at some fixed point $\textrm{v} \in \cG$. In this way, both $\cG$ and $\mathcal{K}_n$ can be seen as subgraphs of $\cG_n$. The most natural example is the case when $\mathcal{K}_n$ consists of a single edge, attached to $\cG$ in a vertex.  

\begin{proposition}\label{thm 4.3 ast}
Let $\mu \in (0,\mu_{\R^+})$ and $\alpha>0$. If $\cE_\a(\mu,\cG_n)$ admits a ground states for every large $n$, and $|\mathcal{K}_n| \to 0$, then also $\cE_\a(\mu, \cG)$ admits a ground state. 
\end{proposition} 

Once that this result is proved, Proposition \ref{prop: non ex tip} follows easily.

\begin{proof}[Proof of Proposition \ref{prop: non ex tip}]
Suppose by contradiction that $\cE_\a(\mu, \cG_{\ell})$ admits a ground state for a sequence $\ell \to 0^+$. Then, by Proposition \ref{thm 4.3 ast} also $\cE_\alpha(\mu, \cG_0)$ would have a ground state, in contradiction with the assumption.
\end{proof}

In proving Proposition \ref{thm 4.3 ast}, we adapt the argument from \cite{AST3}, and to this purpose we shall need some uniform (with respect to $\cG$) estimates for minimizing sequences. In \cite{AST3}, similar estimates follow by exploiting suitable scaling; here we cannot adapt this strategy and hence we argue in a different way.

\begin{lemma}\label{lem: unif est}
Let $\cG$ be a non-compact graph, $\mu \in (0,\mu_{\R^+})$ and $\alpha>0$. Then there exists $\bar C>0$ depending only on $\alpha$ and $\mu$ (in particular, $\bar C$ is independent of $\cG$) such that
\[
\quad \lp{u}{p} + \lp{u}{6} + \lp{u}{\infty} + \lp{u'}{2} \le 4\bar C
\]  
for every $u \in H^1_\mu(\cG)$ such that $E_\a(\mu, \cG) \le \frac12 \cE_\a(\mu,\R)$.
\end{lemma} 

\begin{proof}
Let $u \in H^1_\mu(\cG)$ be such that $E_\a(\mu, \cG) \le \frac12 \cE_\a(\mu,\R)$. From the Gagliardo-Nirenberg inequality, we have that
\[
\frac12 \left( 1 -\left(\frac{\mu}{\mu_{\cG}}\right)^2 \right) \lp{u'}{2}^2- \frac{\alpha C_p(\cG)}{p} \mu^{\frac{p+2}{4}}\lp{u'}{2}^{\frac{p-2}2} \le  \frac12 \cE_\a(\mu,\R).
\]
Recalling \eqref{GN cost} and \eqref{crit mass}, this implies that
\[
\frac12 \left( 1 -\left(\frac{\mu}{\mu_{\R^+}}\right)^2 \right) \lp{u'}{2}^2 \le \frac12 \cE_\a(\mu,\R) + \frac{\alpha C_p(\R^+)}{p} \mu^{\frac{p+2}{4}}\lp{u'}{2}^{\frac{p-2}2}.
\]
Therefore there exists $\bar C>0$ depending only on $\alpha$ and $\mu$ such that $\lp{u'}{2} \le \bar C$. By using again the Gagliardo-Nirenberg inequality and \eqref{GN cost} we directly infer that (if necessary replacing $\bar C$ with a larger quantity) also $\lp{u}{q} \le \bar C$, for both $q=p$ and $q=6$. Finally, the fact that $\lp{u}{\infty} \le \bar C$ is a consequence of the estimate $\lp{u}{\infty} \le 2 \mu \lp{u'}{2}$.
\end{proof}

\begin{proof}[Proof of Proposition \ref{thm 4.3 ast}]
With Lemma \ref{lem: unif est} in our hands, we can extend the proof of \cite[Theorem 4.3]{AST3} with minor changes. We only report a brief sketch for the sake of completeness. Suppose by contradiction that $\cE_\alpha(\mu,\cG_n)$ admits a ground state $u_n$ for every $n$, but $\cE_\alpha(\mu, \cG)$ does not. Then, by \eqref{dis gs} (and recalling that $\cE_\a(\mu,\R)<0$), we know that
\[
E_\alpha(u_n, \cG_n) = \cE_{\alpha}(\mu, \cG_n) \le \cE_\a(\mu,\R) < \frac12 \cE_\a(\mu,\R).
\]
Also, by Theorem \ref{thm: main foc} and \eqref{dis gs}, $\cE_\a(\mu,\cG) = \cE_\a(\mu, \R)$. Let now $\sigma_n := \int_{\cG} u_n^2$, and define $v_n := (\mu/\sigma_n)^{1/2} u_n|_{\cG}$ (where $u_n|_{\cG}$ denotes the restriction of $u_n$ on $\cG$). If $\sigma_n=\mu$, then $u_n$ can be regarded as a function in $H^1_\mu(\cG)$ at level $E_\alpha(u_n, \cG) = \cE_\alpha(\mu, \cG_n) \le \cE_\alpha(\mu, \cG)$, i.e. $u_n$ would be a ground state on $\cG$, a contradiction. Then $\sigma_n<\mu$, $v_n \in H^1_\mu(\cG)$, and arguing as in \cite[Eq. (33)]{AST3} we deduce that
\beq\label{33ast3}
 \cE_\a(\mu, \R) < E_\alpha(v_n, \cG) \le \frac{\mu}{\sigma_n} \left( \cE_\a(\mu,\R) + \frac16 \int_{\mathcal{K}_n} |u_n|^6 + \frac{\alpha}p \int_{\mathcal{K}_n} |u_n|^p\right).
\eeq
Now, by Lemma \ref{lem: unif est} and the fact that $|\mathcal{K}_n| \to 0$, we obtain $\int_{\mathcal{K}_n} |u_n|^q \to 0$ for $q=2$, $q=p$, and $q=6$; also, $\sigma_n \to \mu$, and hence, by \eqref{33ast3}, $v_n$ is a minimizing sequence for $\cE_\a(\mu, \cG) = \cE_\a(\mu, \R)$. Since $\cE_\a(\mu, \cG)$ has no ground state, alternative ($i$) in Lemma \ref{lem: compactness} holds for $\{v_n\}$, whence it follows that $u_n \to 0$ in $L^\infty_{\loc}(\cG)$. As in \cite[Eq. (34)]{AST3}, this implies that $M_n := \|u_n\|_{L^{\infty}(\mathcal{K}_n)} \to 0$ as well, and coming back to \eqref{33ast3} we find that
\[
 (\mu-\sigma_n)(-\cE_\a(\mu, \R)) < \mu \left( \frac{M_n^4}{6} + \frac{\alpha M_n^{p-2}}{p}\right) (\mu-\sigma_n).
\]
Since $\sigma_n<\mu$ and $M_n \to 0$, the estimate yields $\cE_\a(\mu, \R) \ge 0$, which is a contradiction.
\end{proof}

\section{Proof of the main result in the defocusing cases $\alpha<0$}\label{sec: def}

In proving Theorem \ref{thm: main def}, we shall take advantage of what we observed in Remark \ref{rmk: main def}. In particular, we know that alternative ($i$) in the theorem applies to graphs of type (1) - with a terminal edge, while alternative ($ii$) concerns graphs of type (3) and (4) with $\mu_{\cG}<\mu_{\R}$.   

\begin{proof}[Proof of Theorem \ref{thm: main def} - ($i$)]
Let $\mu > \mu_{\R^+}$, and let $\ell$ denote the length of the terminal edge. Then there exists $u \in H^{1}_\mu(\R^+)$ with compact support contained in $[0,M]$ (for some positive $M$), such that $E_0(u,\R^+) <0$. The function $u_\lambda(x) = \sqrt{\lambda} u(\lambda x)$ is in $H^{1}_\mu(\R^+)$, and has compact support in $[0, M/\lambda]$. Thus, for every $\lambda$ sufficiently large, we can consider $u_\lambda$ as an element of $H^1_\mu(\cG)$, with support entirely contained on the terminal edge (and extended to $0$ outside). The energy of $u_\lambda$ is
\[
E_{\alpha}(u_\lambda, \cG) = E_0(u_\lambda,\cG) - \frac{\alpha}{p} \int_{\cG} |u_\lambda|^p = \lambda^2 E_0(u,\R^+) - \frac{\alpha}{p} \lambda^{\frac{p-2}{2}} \int_0^{+\infty} |u|^p,
\]
which tends to $-\infty$ as $\lambda \to +\infty$, since $p<6$.
\end{proof}

Now we focus on alternative ($ii$) in Theorem \ref{thm: main def}, supposing that $\mu_{\cG}<\tilde \mu_{\cG}$. This assumption is fulfilled by graphs of type (3) and some graphs of type (4). We start with a preliminary lemma.

\begin{lemma}\label{lem: main def}
Let $\cG$ be a non-compact metric graph without terminal edges, $p \in (2,6)$, $\alpha<0$. Suppose that $\mu_{\cG}<\mu_{\R}$, and let $\mu \in (\mu_{\cG}, \mu_{\R}]$. Then there exists $\bar \alpha<0$ depending on $p$, $\mu$ and $\cG$ (possibly equal to $-\infty$) such that
\[
\cE_{\alpha}(\mu, \cG) < 0 \quad \iff \quad  \alpha \in ( \bar \alpha,0), \quad \text{and} \quad \cE_{\alpha}(\mu, \cG) =0 \quad \iff \quad  \alpha \le \bar \alpha.
\]
Moreover, $\bar \alpha>-\infty$ if $\mu<\mu_{\R}$.
\end{lemma}

\begin{proof}
Since $\mu \le \mu_{\R}$, by \eqref{dis gs} and Theorem \ref{thm: R} we have that $\cE_\a(\mu, \cG) \le 0$, for every $\alpha<0$. Moreover, by monotonicity of the energy with respect to $\alpha$, we have that $\alpha \mapsto \cE_\a(\mu, \cG)$ is monotone non-increasing. Let us define
\beq\label{def alpha bar}
\bar\alpha:= \sup\{\alpha<0: \ \cE_\a(\mu, \cG) =0\}.
\eeq
We claim that $\bar \alpha <0$. This follows from the continuity of $E_\a$ with respect to $\alpha$: if $\alpha = 0$, since $\mu > \mu_{\cG}$, there exists $u \in H^1_\mu(\cG)$ such that $E_0(u,\cG) <0$ \cite[Proposition 2.4]{AST2}. But then $E_\a(u,\cG) <0$ for small $|\alpha|$, which implies that $\cE_\a(\mu, \cG)<0$ for any such $\alpha$. Now, by definition of $\bar \alpha$ and monotonicity of $\cE_\a(\mu, \cG)$ with respect to $\alpha$, we have that $\cE_\a(\mu, \cG) <0$ for $\bar \alpha<\alpha<0$, and $\cE_\a(\mu, \cG) =0$ for $\alpha<\bar \alpha$. Furthermore, again by continuity of the energy with respect to $\alpha$, it is also not difficult to check that, if $\bar \alpha>-\infty$, then necessarily $\cE_{\bar \alpha}(\mu, \cG) =0$. 

It remains to show that $\bar \alpha >-\infty$ for every $\mu<\mu_{\R}$. To this end, we fix $\mu<\mu_{\R}$ and suppose that there exists $u \in H^1_\mu(\cG)$ such that $E_\a(\mu, \cG)  < 0$. From this, we shall derive an upper bound on $|\alpha|$. By the modified Gagliardo-Nirenberg inequality in Lemma \ref{lem: mod GN}, we deduce that
\[
\begin{split}
\frac{1}{2} \left(1-\left(\frac{\mu}{\mu_{\R}}\right)^2\right)  & \lp{u'}{2}^2- C_{\cG} \mu^\frac12  \\
& \le \frac{1}{2} \left(1-\left(\frac{\mu}{\mu_{\R}}\right)^2\right) \lp{u'}{2}^2- C_{\cG} \mu^\frac12 + \frac{|\alpha|}{p} \lp{u}{p}^p  \le E_\a(u,\cG) < 0,
\end{split}
\]
with $C_{\cG}>0$ depending only on $\cG$. Since $\mu<\mu_{\R}$, it follows that $\lp{u'}{2} \le C(\mu, \cG)$\footnote{Here and in the rest of the proof, $C(\mu, \cG)$ and $C(p,\mu, \cG)$ denote positive constants depending on $\mu$ and $\cG$ or on $p$, $\mu$ and $\cG$, whose precise value may change from line to line.}. But then $\lp{u}{\infty} \le 2 \lp{u}{2} \lp{u'}{2} \le C(\mu, \cG)$, and in turn
\[
\lp{u}{6}^{6} = \int_{\cG} |u|^p |u|^{6-p} \le \lp{u}{\infty}^{6-p} \lp{u}{p}^p \le C(p,\mu, \cG) \lp{u}{p}^p.
\]
Coming back to the fact that $E_\alpha(u, \cG) <0$, the above estimate yields
\[
\frac{|\alpha|}p\lp{u}{p}^p < \frac12\lp{u'}{2}^2 + \frac{|\alpha|}p\lp{u}{p}^p < \frac16\lp{u}{6}^6 \le C(p,\mu, \cG) \lp{u}{p}^p,
\]
whence necessarily
\[
\left( \frac{|\alpha|}{p} - C(p,\mu, \cG) \right) \lp{u}{p}^p <0.
\]
This is possible only if $|\alpha|< C(p,\mu, \cG)$. In other words, if $|\alpha|$ is greater than a constant depending on $p$, $\mu$ and $\cG$, then $\cE_\alpha(\mu, \cG) = 0$, so that $\bar \alpha>-\infty$, and the proof is complete. 
\end{proof}

The next lemmas will allow us to prove the existence of ground states when $\alpha \in (\bar \alpha,0)$.

\begin{lemma}\label{lem: conv se non 0}
Let $\cG$ be a non-compact metric graph without terminal edges, $p \in (2,6)$, $\alpha<0$. Suppose that $\mu_{\cG}<\mu_{\R}$, and let $\mu \in (\mu_{\cG}, \mu_{\R}]$. Let $\{u_n\}$ be a minimizing sequence of non-negative functions for $\cE_\a(\mu, \cG)$, and suppose that $u_n \weak u$ weakly in $H^1(\cG)$. \\
If $u \not \equiv 0$, then $u \in H^1_\mu(\cG)$ and is a ground state.
\end{lemma}

\begin{proof}
Up to a subsequence, we can suppose that $u_n \to u$ in $L^\infty_{\loc}(\cG)$ and almost everywhere. Thus
\[
E_\a(u_n,\cG) = E_\alpha(u, \cG) + E_\alpha(u_n-u, \cG) + o(1) \ge E_\alpha(u, \cG) + E_0(u_n-u, \cG) + o(1)
\]
as $n \to \infty$, where we used the Brezis-Lieb lemma and the fact that $\alpha<0$. Now, $u_n-u \weak 0$ weakly in $H^1(\cG)$, and hence by \cite[Lemma 4.2]{AST2} we infer that
\[
E_0(u_n-u, \cG) \ge \frac12 \left( 1- \frac{\|u_n-u\|_{L^2(\cG)}^4}{\mu_\R^2}\right) \|u_n'-u'\|_{L^2(\cG)} + o(1) \ge o(1)
\]
as $n \to \infty$, since $\lp{u_n-u}{2}^2 = \mu - \lp{u}{2}^2 +o(1) \ge o(1)$, by weak convergence. This implies that $E_\a(u_n,\cG) \ge  E_\alpha(u, \cG) + o(1)$, so that 
\[
E_\a(u, \cG) \le \cE_\a(\mu, \cG) \le 0.
\]
Notice that by weak convergence $m:= \lp{u}{2}^2\le \mu$, and $m >0$ by assumption. If $m < \mu$, then
\[
\begin{split}
E_\a\left( \sqrt{\frac{\mu}{m}} u, \cG\right) &= \frac{\mu}{m} \frac12 \int_{\cG} |u'|^2 -  \left(\frac{\mu}{m}\right)^3 \frac16 \int_{\cG} |u|^6 - \left( \frac{\mu}{m}\right)^\frac{p}{2} \frac{\alpha}p \int_{\cG} |u|^p \\
& < \left(\frac{\mu}{m}\right)^3 E_\a(u, \cG) \le \cE_\a(\mu, \cG),
\end{split}
\]
since $\mu/m>1$, $\lp{u}{p}>0$, and $\cE_\a(\mu, \cG)\le 0$. This is in contradiction with the definition of $\cE_\a(\mu, \cG)$, and hence necessarily $m=\mu$, that is $u \in H^1_\mu(\cG)$ and is the desired ground state (one can also deduce that $u_n \to u$ strongly in $H^1(\cG)$: since $m=\mu$, $u_n \to u$ strongly in $L^2(\cG)$; this and the boundedness of $\{u_n\}$ in $H^1$ imply that $u_n \to u$ strongly in $L^p(\cG)$ and in $L^6(\cG)$. Then, using that $E_\a(u, \cG) = \lim_n E_\a(u_n, \cG)$, the strong convergence in $H^1(\cG)$ follows).
\end{proof}

%\begin{remark}
%It is natural to conjecture that $\bar \alpha_1= \bar \alpha_2$. This seems to be related to the existence of ``extremals" for the modified Gagliardo-Nirenberg inequality, which seems to be a difficult task due to the presence of $\theta_u$ in the inequality.
%\end{remark}

With Lemmas \ref{lem: main def} and \ref{lem: conv se non 0} in our hands, we can proceed with the:
\begin{proof}[Proof of Theorem \ref{thm: main def}-($ii$)] 
The estimates on $\cE_\a(\mu, \cG)$ follows directly from Lemma \ref{lem: main def}. 

Now we show that there are no ground states if $\alpha< \bar \alpha$. Suppose by contradiction that $\cE_\alpha(\mu, \cG) = 0$ is achieved by some $u_\alpha \in H^1_\mu(\cG)$, for some $\alpha< \bar \alpha$. Let $\beta \in (\alpha, \bar \alpha)$. Then $u_\alpha \not \equiv 0$, and 
\[
E_\beta(u_\alpha, \cG) = E_\a(u_\a,\cG) + \frac{\alpha-\beta}{p} \int_{\cG} u_\alpha^p < E_\a(u_\a,\cG) = 0,
\]
in contradiction with the fact that $\cE_\beta(\mu,\cG) =0$, since $\beta< \bar \alpha$.

Now we focus on the case $\alpha \in (\bar \alpha,0)$. Let $u \in H^1_\mu(\cG)$ be such that $E_\a(u, \cG) <\cE_\a(\mu, \cG)/2<0$. Then, by Lemma \ref{lem: mod GN}, there exists $\theta_u \in [0,\mu]$ such that
\[
\lp{u}{6}^6 \le 3 \left(1- \frac{\theta_u}{\mu_{\R}}\right)^2 \lp{u'}{2}^2 + C_{\cG} \theta^\frac12,
\]
with $C_{\cG}>0$ depending only on $\cG$. Therefore
\[
 \frac{\theta_u}{2\mu_{\R}} \left( 2- \frac{\theta_u}{\mu_{\R}}\right) \lp{u'}{2}^2 - C_{\cG} \theta_u^{\frac{1}{2}} + \frac{|\alpha|}{p} \lp{u}{p}^p \le E_\a(u, \cG) < \frac12 \cE_\a(\mu, \cG),
\]
whence there exists $c=c(\alpha,\mu,\cG)>0$ such that $\theta_u \ge c$, for every $u \in H^1_\mu(\cG)$ such that $E_\a(u, \cG) <\cE_\a(\mu, \cG)/2$. In turn, this implies that for any such $u$ 
\[
\frac{c}{2\mu_{\R}} \lp{u'}{2}^2 - C_{\cG} \mu^{\frac{1}{2}} <0,
\]
and hence there exists $C=C(\alpha,\mu,\cG)>0$ such that $\lp{u'}{2} \le C$, for every $u \in H^1_\mu(\cG)$ such that $E_\a(u, \cG) <\cE_\a(\mu, \cG)/2$. In particular, any minimizing sequence for $\cE_\a(\mu, \cG)$ is bounded in $H^1(\cG)$, and up to a subsequence $u_n \weak u$ weakly in $H^1(\cG)$. If the weak limit is $0$, then by using the fact that $\alpha<0$ and \cite[Lemma 4.2]{AST2}, we have that
\[
E_\alpha(u_n, \cG) \ge E_0(u_n,\cG) \ge \frac12 \left( 1-\left(\frac{\mu}{\mu_{\R}}\right)^2 \right) \lp{u_n'}{2}^2 +o(1) \ge o(1)
\]
as $n \to \infty$, since $\mu \le \mu_{\R}$. But then $\cE_\a(\mu, \cG) \ge 0$, which is in contradiction with the choice $\alpha \in (\bar \alpha,0)$. Therefore, the weak limit $u \not \equiv 0$, and $u$ is the desired ground state by Lemma \ref{lem: conv se non 0}.
\end{proof}

We conclude this section with the:

\begin{proof}[Proof of Proposition \ref{prop: non ex star}]
Suppose by contradiction that there exists a critical point $u \in H^1_\mu(\cG_N)$ of $E_\a(\cdot, \cG_N)$ on $H^1_\mu(\cG_N)$, for some $\mu \in (0, \mu_{\R}]$. Then $u$ is identified by a vector $u= (u_1,\dots,u_N)$ in $H^1(\R^+, \R^N)$, with $u_1(0) = \dots = u_N(0)$, such that for some $\lambda \in \R$
\[
\begin{cases}
-u_i''+ \lambda u_i = |u_i|^4 u_i + \alpha |u_i|^{p-2} u_i & \text{on $(0,+\infty)$, for every $i=1,\dots, N$} \\
\sum_{i=1}^N u_i'(0^+) = 0
\end{cases}
\]
(see e.g. \cite[Proposition 3.3]{AST1} for the details). By testing the equation for $u_i$ against $u_i$ on $(0,r)$, taking the limit along a sequence $r \to +\infty$ with $\lim_{r \to +\infty} u_i'(r) = 0$ (such a sequence does exist, since $u_i \in C^1(0,+\infty)$ tends to $0$ at infinity), and summing over $i=1,\dots,N$, we obtain
\[
\int_{\cG_N} |u'|^2 + \lambda |u|^2 = \int_{\cG_N} |u|^6 + \alpha |u|^p.
\]
In a similar way, by testing the equation for $u_i$ against $x u_i'$ on $(0,r)$, taking the limit along a suitable sequence $r \to +\infty$ (as in \cite[Proposition 1]{BeLi}), and summing over $i=1,\dots,N$, we also obtain the validity of the Pohozaev identity
\[
\int_{\cG_N} -|u'|^2 + \lambda |u|^2 = \int_{\cG_N} \frac13|u|^6 + \frac2p \alpha |u|^p.
\]
By combining the previous identities, we infer that 
\beq\label{Poho}
\int_{\cG_N} |u'|^2 = \int_{\cG_N} \frac13 |u|^6 +\alpha \frac{p-2}{2p} |u|^p.
\eeq
As a consequence, recalling that the star graph admits a cycle covering, and hence $\mu_{\cG}=\mu_{\R}$, we deduce that
\[
0 \le E_0(u, \cG_N) = \int_{\cG_N} \frac12|u'|^2 - \frac16|u|^6 = \frac{\alpha}2 \frac{p-2}{2p} \int_{\cG}|u|^p <0,
\]
since $\alpha<0$. This contradiction shows that there is no critical point of $E_\a(\cdot, \cG_N)$ on $H^1_\mu(\cG_N)$, when $\mu \in (0,\mu_{\R}]$.

Concerning the possible existence of local minimizers for $E_\a(\cdot, \cG_N)$ on $H^1_\mu(\cG_N)$, with arbitrary $\mu>0$, let us assume by contradiction that such a local minimizer $u$ does exist. As previously explained, $u$ is identified by a vector $u=(u_1,\dots,u_N) \in H^1_\mu(\R^+, \R^N)$, with $u_1(0) = \dots=u_N(0)$. It is immediate to check that $u_\lambda=(u_{1,\lambda}, \dots,u_{N,\lambda})$ defined by $u_{i,\lambda}(x) = \sqrt{\lambda} u_{i}(\lambda x)$ is another function in $H^1_\mu(\cG_N)$, and that
\[
E_\a(u_\lambda, \cG_N) = \lambda^2 \int_{\cG_N} \left(\frac12|u'|^2- \frac16 |u|^6\right)   - \frac{\alpha}p \lambda^\frac{p-2}2 \int_{\cG_N} |u|^p.
\]
By local minimality, we should have that 
\[
\left. \frac{d}{d\lambda} E_\a(u_\lambda, \cG_N)\right|_{\lambda = 1} = 0, \qquad \left. \frac{d^2}{d\lambda^2} E_\a(u_\lambda, \cG_N)\right|_{\lambda = 1} \ge 0.
\]
However, while the condition of the first derivative is satisfied (by the Pohozaev identity \eqref{Poho}), the condition on the second derivative yields
\[
\begin{split}
0 & \le 2 \int_{\cG_N} \left(\frac12|u'|^2- \frac16 |u|^6\right) - \frac{\alpha}{p} \left(\frac{p-2}{2}\right)\left(\frac{p-4}{2}\right) \int_{\cG_N} |u|^p \\
& = \frac{\alpha}{p} \left(\frac{p-2}{2}\right)\left(1- \frac{p-4}{2}\right) \int_{\cG_N} |u|^p = \frac{\alpha}{p} \left(\frac{p-2}{2}\right)\left(\frac{6- p}{2}\right) \int_{\cG_N} |u|^p<0,
\end{split}
\]
since $\alpha<0$ and $p \in (2,6)$. This is the desired contradiction.
\end{proof}

%\bibliography{normalized}
%\bibliographystyle{abbrv}

\end{document}